\documentclass[10pt]{article}
\usepackage[latin1]{inputenc}
\usepackage{epsfig}
\usepackage{color}
\usepackage[british,english]{babel}
\usepackage{amsthm}
\usepackage{amsmath}
\usepackage{amsfonts}
\usepackage{amssymb}
\usepackage{graphicx}
\def \noame{\noalign{\medskip}}
\newtheorem{corollary}{Corollary}[section]

\newtheorem{lemma}[corollary]{Lemma}
\newtheorem{proposition}[corollary]{Proposition}
\newtheorem{remark}[corollary]{Remark}
\newtheorem{theorem}[corollary]{Theorem}
\newfont{\sBlackboard}{msbm10 scaled 900}

\newcommand{\mylabel}[1]{\label{#1}
            \ifx\undefined\stillediting
            \else \fbox{$#1$}\fi }
\newcommand{\BE}{\begin{equation}}

\newcommand{\EEQ}{\end{equation}}
\newcommand{\rfb}[1]{\mbox{\rm
   (\ref{#1})}\ifx\undefined\stillediting\else:\fbox{$#1$}\fi}

\newfont{\Blackboard}{msbm10 scaled 1200}

\newfont{\roma}{cmr10 scaled 1200}

\def\CC{\rm \hbox{C\kern-.56em\raise.4ex
         \hbox{$\scriptscriptstyle |$}\kern+0.5 em }}

\def\n{|\kern -.05cm{|}\kern -.05cm{|}}

\newcommand{\mm}    {{\hbox{\hskip 0.5pt}}}

\newcommand{\bluff} {{\hbox{\raise 15pt \hbox{\mm}}}}
\usepackage{fancyhdr}

\makeatletter
\def\section{\@startsection {section}{1}{\z@}{-3.5ex plus -1ex minus
    -.2ex}{2.3ex plus .2ex}{\large\bf}}
\makeatother
\def\be{\begin{equation}}
\def\ee{\end{equation}}

\date{ }
\begin{document}
\thispagestyle{empty}
\title{\bf Newtonian fluid flow in a thin porous medium with non-homogeneous slip boundary conditions}\maketitle

\author{ \center  Mar\'ia ANGUIANO\\
Departamento de An\'alisis Matem\'atico. Facultad de Matem\'aticas.\\
Universidad de Sevilla, P. O. Box 1160, 41080-Sevilla (Spain)\\
anguiano@us.es\\}
\medskip\author{ \center  Francisco Javier SU\'AREZ-GRAU\\ Departamento de Ecuaciones Diferenciales y An\'alisis Num\'erico. Facultad de Matem\'aticas. \\ Universidad de Sevilla, 41012-Sevilla (Spain)\\
 fjsgrau@us.es\\}

\vskip20pt

 \renewcommand{\abstractname} {\bf Abstract}
\begin{abstract} 
We consider the Stokes system in a thin porous medium $\Omega_\varepsilon$ of thickness $\varepsilon$ which is perforated by periodically distributed solid cylinders of size $\varepsilon$. On the boundary of the cylinders we prescribe non-homogeneous slip boundary conditions depending on a parameter $\gamma$. The aim is to give the asymptotic behavior of the velocity and the pressure of the fluid as $\varepsilon$ goes to zero. Using an adaptation of the unfolding method, we give, following the values of $\gamma$, different limit systems.
\end{abstract}
\bigskip\noindent

 {\small \bf AMS classification numbers:} 76A20, 76M50, 35B27.  \\
 
   \newpage

\section {Introduction}\label{S1}
We consider a viscous fluid obeying the Stokes system in a thin porous medium $\Omega_\varepsilon$ of thickness $\varepsilon$ which is perforated by periodically distributed  cylinders (obstacles) of size $\varepsilon$. On the boundary of the obstacles, we prescribe a Robin-type condition depending on a parameter $\gamma \in \mathbb{R}$. The aim of this work is to prove the convergence of the homogenization process when $\varepsilon$ goes to zero, depending on the different values of $\gamma$.

\paragraph*{The domain:} the periodic porous medium is defined by a domain $\omega$ and an associated microstructure, or periodic cell $Y^{\prime}=[-1/2,1/2]^2$, which is made of two complementary parts: the fluid part $Y^{\prime}_{f}$, and the obstacle part $T^{\prime}$ ($Y^{\prime}_f  \bigcup T^{\prime}=Y^\prime$ and $Y^{\prime}_f  \bigcap T^{\prime}=\varnothing$). More precisely, we assume that $\omega$ is a smooth, bounded, connected set in $\mathbb{R}^2$, and that $T^{\prime}$ is an open connected subset of $Y^\prime$ with a smooth boundary $\partial T^\prime$, such that $\overline T^\prime$ is strictly included  in $Y^\prime$.\\

The microscale of a porous medium is a small positive number ${\varepsilon}$. The domain $\omega$ is covered by a regular mesh of square of size ${\varepsilon}$: for $k^{\prime}\in \mathbb{Z}^2$, each cell $Y^{\prime}_{k^{\prime},{\varepsilon}}={\varepsilon}k^{\prime}+{\varepsilon}Y^{\prime}$ is divided in a fluid part $Y^{\prime}_{f_{k^{\prime}},{\varepsilon}}$ and an obstacle part $T^{\prime}_{k^{\prime},{\varepsilon}}$, i.e. is similar to the unit cell $Y^{\prime}$ rescaled to size ${\varepsilon}$. We define $Y=Y^{\prime}\times (0,1)\subset \mathbb{R}^3$, which is divided in a fluid part $Y_{f}=Y'_f\times (0,1)$ and an obstacle part $T=T'\times(0,1)$, and consequently $Y_{k^{\prime},{\varepsilon}}=Y^{\prime}_{k^{\prime},{\varepsilon}}\times (0,1)\subset \mathbb{R}^3$, which is also divided in a fluid part $Y_{f_{k^{\prime}},{\varepsilon}}$ and an obstacle part $T_{{k^{\prime}},{\varepsilon}}$. 

We denote by $\tau(\overline T'_{k',\varepsilon})$ the set of all translated images of $\overline T'_{k',\varepsilon}$. The set $\tau(\overline T'_{k',\varepsilon})$ represents the obstacles in $\mathbb{R}^2$.

\begin{figure}[h!]
\begin{center}
\includegraphics[width=3cm]{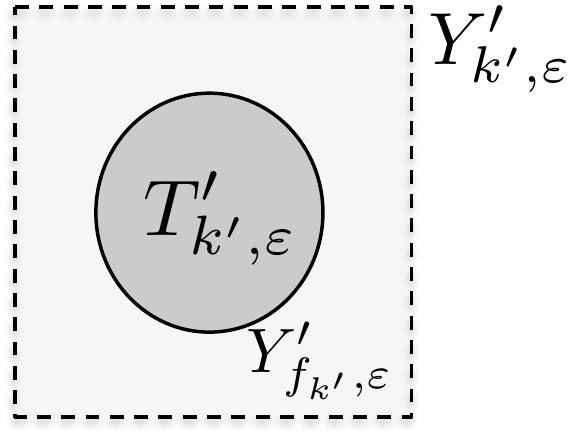} \quad\quad\quad\quad
\includegraphics[width=4cm]{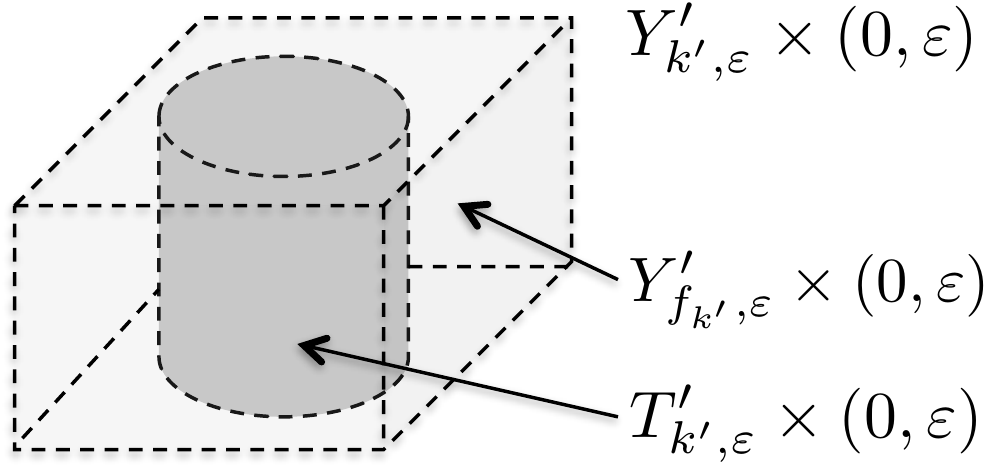}
\end{center}
\caption{Views of a periodic cell in 2D (left) and 3D (right)}
\label{fig:periodic_cell_2D_3D}
\end{figure}

The fluid part of the bottom $\omega_{\varepsilon}\subset \mathbb{R}^2$ of the porous medium is defined by $\omega_{\varepsilon}=\omega\backslash\bigcup_{k^{\prime}\in \mathcal{K}_{\varepsilon}} \overline T^{\prime}_{{k^{\prime}},{\varepsilon}},$ where $\mathcal{K}_{\varepsilon}=\{k^{\prime}\in \mathbb{Z}^2: Y^{\prime}_{k^{\prime}, {\varepsilon}} \cap \omega \neq \emptyset \}$.   The whole fluid part $\Omega_{\varepsilon}\subset \mathbb{R}^3$ in the thin porous medium is defined by (see Figures  \ref{fig:lateral_above}  and \ref{fig:domain})
\begin{equation}\label{Dominio1}
\Omega_{\varepsilon}=\{  (x_1,x_2,x_3)\in \omega_{\varepsilon}\times \mathbb{R}: 0<x_3<\varepsilon \}.
\end{equation}

\begin{figure}[h!]
\begin{center}
\includegraphics[width=7cm]{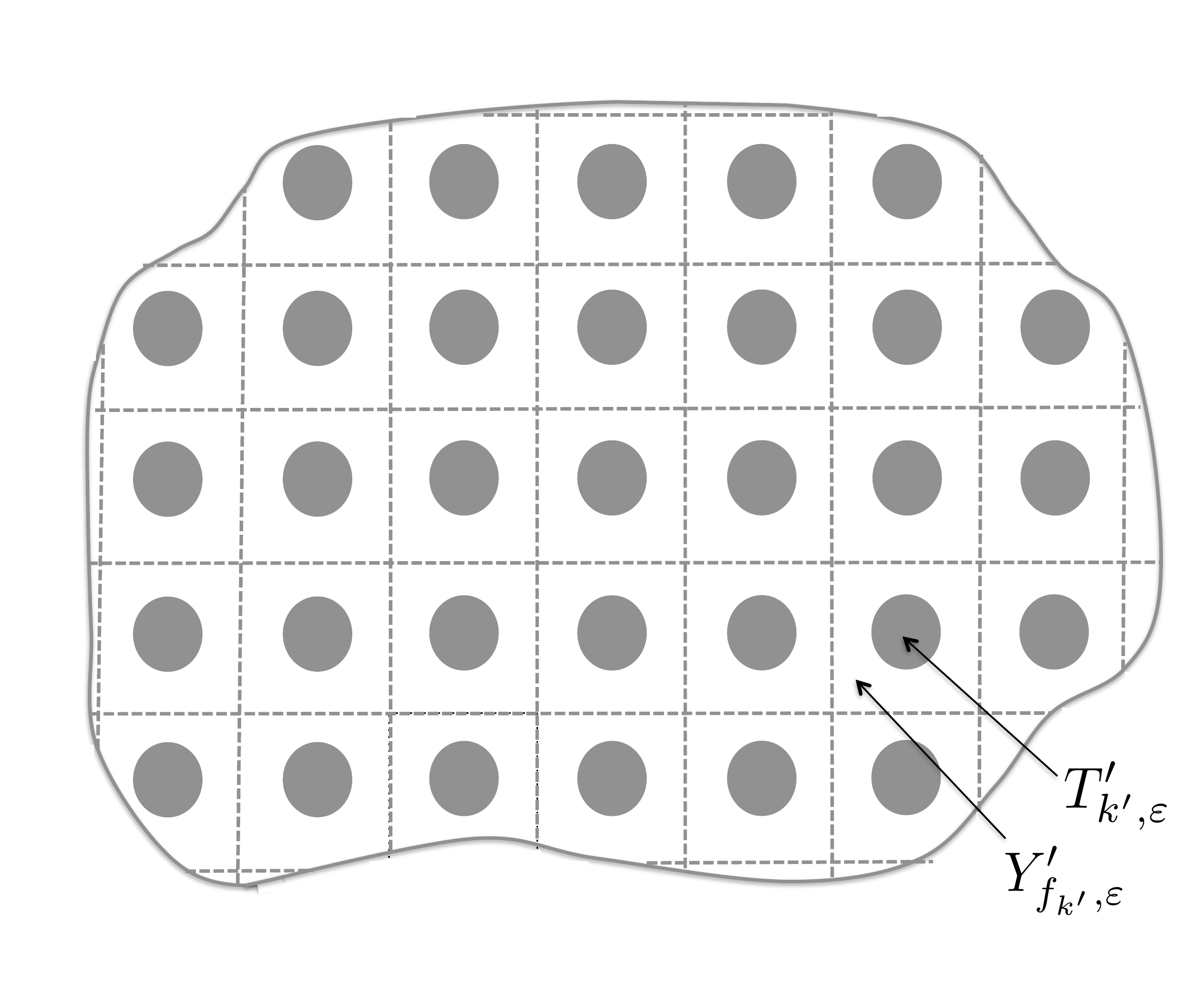}
\end{center}
\caption{View from above}
\label{fig:lateral_above}
\end{figure}

\begin{figure}[h!]
\begin{center}
\includegraphics[width=7cm]{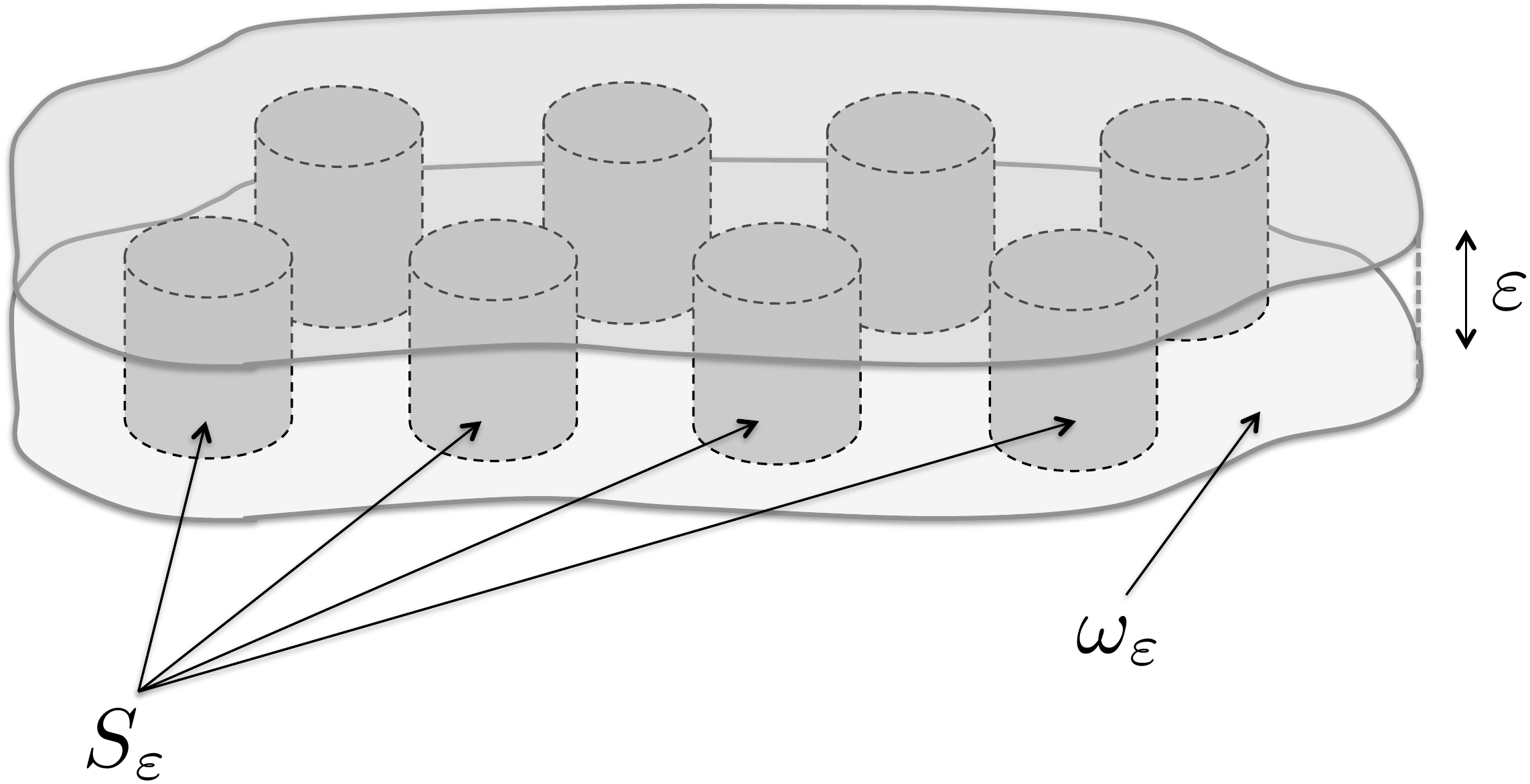}\quad \quad\quad 
\includegraphics[width=7cm]{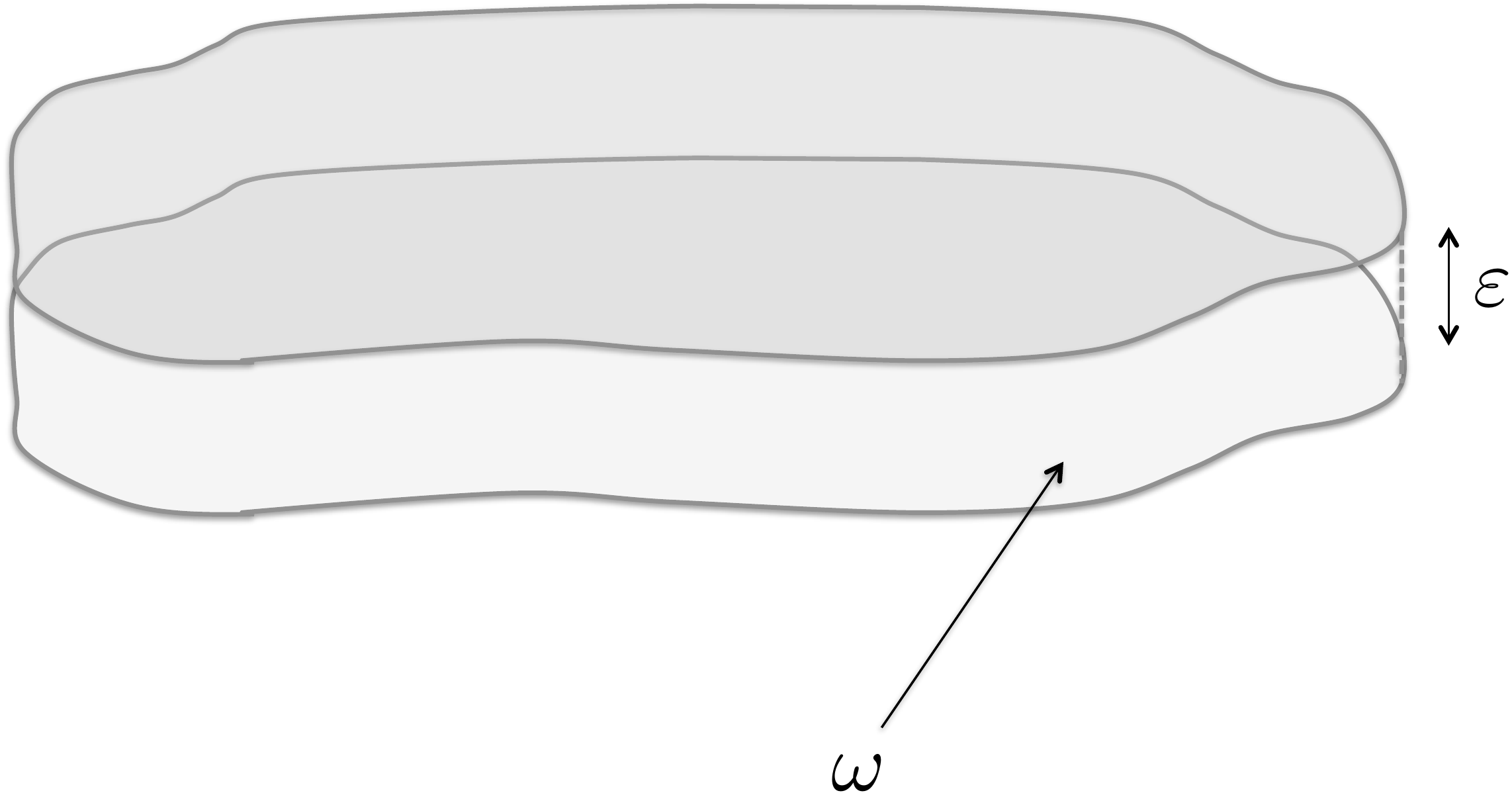}
\end{center}
\caption{Views of the domain $\Omega_\varepsilon$ (left) and $\Lambda_\varepsilon$ (right)}
\label{fig:domain}
\end{figure}

We make the following assumption:
$$\text{The obstacles } \tau(\overline T'_{k',\varepsilon}) \text{ do not intersect the boundary } \partial \omega.$$

We define $T^\varepsilon_{k',\varepsilon}=T'_{k',\varepsilon}\times (0,\varepsilon)$. Denote by $S_\varepsilon$ the set of the obstacles contained in $\Omega_\varepsilon$. Then, $S_\varepsilon$ is a finite union of obstacles, i.e. $$S_\varepsilon=\bigcup_{k^{\prime}\in \mathcal{K}_{\varepsilon}} \overline T^\varepsilon_{k',\varepsilon}.$$

We define 
\begin{equation}\label{OmegaTilde}
\widetilde{\Omega}_{\varepsilon}=\omega_{\varepsilon}\times (0,1), \quad \Omega=\omega\times (0,1), \quad \Lambda_\varepsilon=\omega\times (0,\varepsilon).
\end{equation}

We observe that $\widetilde{\Omega}_{\varepsilon}=\Omega\backslash \bigcup_{k^{\prime}\in \mathcal{K}_{\varepsilon}} \overline T_{{k^{\prime}}, {\varepsilon}},$ and we define $T_\varepsilon=\bigcup_{k^{\prime}\in \mathcal{K}_{\varepsilon}} \overline T_{k^\prime, \varepsilon}$ as the set of the obstacles contained in $\widetilde \Omega_\varepsilon$.

\paragraph*{The problem:} let us consider the following Stokes system in $\Omega_\varepsilon$, with a Dirichlet boundary condition on the exterior boundary $\partial \Lambda_\varepsilon$ and a non-homogeneous slip boundary condition on the cylinders $\partial S_\varepsilon$:
\begin{equation}
\left\{
\begin{array}
[c]{r@{\;}c@{\;}ll}%
\displaystyle -\mu\Delta u_{\varepsilon}+ \nabla p_{\varepsilon} &
= &
f_\varepsilon \quad \text{\ in \ }\Omega_{\varepsilon},\\
\noame
{\rm div}\, u_{\varepsilon} & = & 0 \quad \text{\ in \ }\Omega_{\varepsilon},\\
\noame
u_\varepsilon& = & 0 \quad \text{\ on \ }\partial\Lambda_\varepsilon,\\
\displaystyle -p_\varepsilon\cdot n+\mu{\partial u_{\varepsilon}\over \partial n} + \alpha \varepsilon^\gamma u_\varepsilon&=&g_\varepsilon \quad \text{\,on \ } \partial S_{\varepsilon},
\end{array}
\right. \label{1}%
\end{equation}
where we denote by $u_\varepsilon=(u_{\varepsilon,1},u_{\varepsilon,2},u_{\varepsilon,3})$ the velocity field, $p_\varepsilon$ is the (scalar) pressure, $f_\varepsilon=(f_{\varepsilon,1}(x_1,x_2),f_{\varepsilon,2}(x_1,x_2),0)$ is the field of exterior body force and $g_\varepsilon=(g_{\varepsilon,1}(x_1,x_2),g_{\varepsilon,2}(x_1,x_2),0)$ is the field of exterior surface force. The constants $\alpha$ and $\gamma$ are given, with $\alpha > 0$, $\mu$ is the viscosity and $n$ is the outward normal to $S_\varepsilon$.

This choice of $f$ and $g$ is usual when we deal with thin domains. Since the thickness  of the domain, $\varepsilon$,  is small then vertical component of the forces can be neglected and, moreover the force can be considered independent of the vertical variable.\\

Problem (\ref{1}) models in particular the flow  of an incompressible viscous fluid through a porous medium under the action of an exterior electric field. This system is derived from a physical model well detailed in the literature. As pointed out in  Cioranescu {\it et al.} \cite{Ciora0} and  Sanchez-Palencia \cite{Sanchez-Palencia},  it was observed experimentally in Reuss \cite{Reuss} the following phenomenon: when a electrical field is applied on the boundary of a porous medium in equilibrium, a motion of the fluid appears.  This motion is a consequence of the electrical field only. To describe such a motion, it is usual to consider a modified Darcy's law considering of including an additional term, the gradient of the electrical field, or consider that  the presence of this term is possible only if the electrical charges have a volume distribution.  However, this law contains implicitly a mistake,  because if the solid and fluid parts are both dielectric, such a distribution does not occur, the electrical charges act only on the boundary between the solid and the fluid parts and so they have necessarily a surface distribution.  If such hypothesis is done, we can describe the boundary conditions in terms of the stress tensor $\sigma_\varepsilon$ as follows
$$\sigma_\varepsilon\cdot n+ \alpha \varepsilon^\gamma u_\varepsilon=g_\varepsilon,$$
which is precisely the non-homogeneous slip boundary condition  (\ref{1})$_4$ and means that the stress-vector $\sigma_\varepsilon\cdot n$ induces a slowing effect on the motion of the fluid, expressed by the term $\alpha\varepsilon^\gamma$. Moreover, if there are exterior forces like for instance, an electrical field, then the non-homogeneity of the boundary condition on the $\partial S_\varepsilon$ is expressed in terms of surface charges contained in $g_\varepsilon$. \\

On the other hand, the behavior of the flow of Newtonian fluids through periodic arrays of cylinders has been studied extensively, mainly because of its importance in many applications in heat and mass transfer equipment. However, the literature on Newtonian thin film fluid flows through periodic arrays of cylinders is far less complete, although these problems have now become of great practical relevance because take place in a number of natural and industrial processes. This includes flow during manufacturing of fibre reinforced polymer composites with liquid moulding processes (see Frishfelds {\it et al.} \cite{Frishfelds}, Nordlund and Lundstrom \cite{Nordlund}, Tan and Pillai \cite{Tan}), passive mixing in microfluidic systems (see Jeon \cite{Jeon}), paper making (see Lundstr${\rm \ddot{o}}$m {\it et al.} \cite{Lundstrom}, Singh {\it et al.} \cite{Singh}), and block copolymers self-assemble on nanometer length scales (see Park {\it et al.} \cite{Park}, Albert  and Epps \cite{Albert}, Farrel {\it et al.} \cite{Farrel}).\\

The Stokes problem in a periodically perforated domain with holes of the same size as the periodic has been treated in the literature. More precisely, the case with Dirichlet conditions on the boundary of the holes was studied by Ene and Sanchez-Palencia \cite{EneSanche}, where the model that describes the homogenized medium is a Darcy's law. The case with non-homogeneous slip boundary conditions, depending on a parameter $\gamma\in\mathbb{R}$, was studied by Cioranescu {\it et al.} \cite{Ciora0}, where using the variational method introduced by Tartar \cite{Tartar}, a Darcy-type law, a Brinkmann-type equation or a Stokes-type equation are obtained depending of the values of $\gamma$.  The Stokes and Navier-Stokes equations in a perforated domain with holes of size $r_\varepsilon$, with $r_\varepsilon \ll \varepsilon$ is considered by Allaire \cite{Allaire0}. On the boundary of the holes, the normal component of the velocity is equal to zero and the tangential velocity is proportional to the tangential component of the normal stress. The type of the homogenized model is determined by the size $r_\varepsilon$, i.e. by the geometry of the domain.\\

The earlier results relate to a fixed height domain. For a thin domain, in \cite{Anguiano} Anguiano and Su\'arez-Grau consider an incompressible non-Newtonian Stokes system, in a thin porous medium of thickness $\varepsilon$ that is perforated by periodically distributed solid cylinders of size $a_\varepsilon$, with Dirichlet conditions on the boundary of the cylinders. 

Using a combination of the unfolding method (see Cioranescu {\it et al.} \cite{Ciora} and Cioranescu {\it et al.} \cite{CioraPorous} for perforated domains) applied to the horizontal variables, with a rescaling on the height variable, and using monotonicity arguments to pass to the limit, three different Darcy's laws are obtained rigorously depending on the relation between $a_\varepsilon$ and $\varepsilon$. We remark that an extension of the unfolding method to evolution problems in which the unfolding method is applied to the spatial variables and not on the time variable was introduced in Donato and Yang \cite{D-ZY1} (see also \cite{D-ZY2}).

 The behavior observed when $a_\varepsilon \approx \varepsilon$  in \cite{Anguiano} has motivated the fact of considering non-homogeneous slip conditions on the boundary of the cylinders. In this sense, our aim in the present paper is to consider a Newtonian Stokes system with the non-homogeneous slip boundary condition  (\ref{1})$_4$ in the thin porous medium described in (\ref{Dominio1}) and we prove the convergence of the homogenization process depending on the different values of $\gamma$. To do that, we have to take into account that the normal component of the velocity on the cylinders is different to zero and the extension of the velocity is no longer obvious. If we consider the Stokes system with Dirichlet boundary condition on the obstacles as in \cite{Anguiano}, the velocity can be easily extended by zero in the obstacles, however in this case we need another kind of extension and adapt it to the case of a thin domain. 

One of the main difficulties in the present paper is to treat the surface integrals. The papers mentioned above about problems with non-homogeneous boundary conditions use a generalization (see Cioranescu and Donato \cite{CioraNew}) of the technique introduced by Vanninathan \cite{vanitatis} for the Steklov problem, which transforms the surface integrals into volume integrals. In our opinion, an excellent alternative to this technique was made possible with the development of the unfolding method (see Cioranescu {\it et al.} \cite{Ciora}), which allows to treat easily the surface integrals. In the present paper, we extend some abstract results for thin domains, using an adaptation of the unfolding method, in order to treat all the surface integrals and we obtain directly the corresponding homogenized surface terms. A similar approach is made by Cioranescu {\it et al.} \cite{CioDonZak1} and Zaki \cite{Zaki} with non-homogeneous slip boundary conditions, and Capatina and Ene \cite{Capatina} with non-homogeneous pure slip boundary conditions for a fixed height domain.  

In summary,  we show that the asymptotic behavior of the system (\ref{1}) depends on the values of $\gamma$:
\begin{itemize}
\item[-] for $\gamma<-1$, we obtain a 2D Darcy type law as an homogenized model. The flow is only driven by the pressure.
\item[-] for $-1\leq \gamma<1$, we obtain a 2D Darcy type law but in this case the flow depends on the pressure, the external body force and the mean value of the external surface force.
\item[-] for $\gamma \ge 1$, we obtain a 2D Darcy type law where the flow is only driven by the pressure with a permeability tensor obtained by means of two local 2D Stokes problems posed in the reference cell with  homogeneous Neumann boundary condition on the reference cylinder.
\end{itemize}
We observe that we have obtained the same three regimes as  in Cioranescu {\it et al.} \cite{CioDonZak1} (see Theorems 2.1 and 2.2), and Zaki \cite{Zaki} (see Theorems 14 and 16). Thus, we conclude that the fact of considering the thin domain does not change the critical size of the parameter $\gamma$, but the thickness of the domain introduces a reduction of dimension of the homogenized models and other consequences. More precisely,  in the cases $\gamma<-1$ and   $-1\leq \gamma<1$,  we obtain the same Darcy type law as in \cite{CioDonZak1, Zaki} with the vertical component of the velocity zero as consequence of the thickness of the domain. The main difference appears in the case $\gamma \ge 1$, in which a 3D Brinkmann or Stokes type law were derived in \cite{CioDonZak1, Zaki} while a 2D Darcy type law is obtained in the present paper.

We also remark the differences with the result obtained in \cite{Anguiano} where Dirichlet boundary conditions are pres\-cribed on the cylinders in the case $a_\varepsilon\approx \varepsilon$. In that case, a 2D Darcy law  as an homogenized model with a permeability tensor was obtained through two Stokes local problems in the reference cell with Dirichlet boundary conditions on the reference cylinder.  Here, we obtain three different homogenized model depending on $\gamma$.  The case $\gamma \ge 1$ gives a similar 2D Darcy type law, but the permeability tensor is obtained through two Stokes local problems with homogeneous Neumann boundary conditions. In the cases $\gamma<-1$ and  $-1\leq \gamma<1$, we obtain a 2D Darcy type law without microstructure.

The paper is organized as follows. We introduce some notations in Section \ref{SNot}. In Section \ref{S2}, we formulate the problem and state our main result, which is proved in Section \ref{S3}. The article closes with a few remarks in Section \ref{sec:conclusions}.

\section{Some notations} \label{SNot} Along this paper, the points $x\in\mathbb{R}^3$ will be decomposed as $x=(x^{\prime},x_3)$ with $x^{\prime}\in \mathbb{R}^2$, $x_3\in \mathbb{R}$. We also use the notation $x^{\prime}$ to denote a generic vector of $\mathbb{R}^2$.\\

 In order to apply the unfolding method, we need the following notation:  for $k'\in \mathbb{Z}^2$, we define $\kappa:\mathbb{R}^2\to \mathbb{Z}^2$ by 
\begin{equation}\label{kappa_fun}
\kappa(x^\prime)=k^\prime\iff x'\in Y^{\prime}_{k^{\prime},1}\,.
\end{equation}
Remark that $\kappa$ is well defined up to a set of zero measure in $\mathbb{R}^2$, which is given by $\cup_{k^\prime\in \mathbb{Z}^2}\partial Y^{\prime}_{k^{\prime},1}$. Moreover, for every $\varepsilon>0$, we have 
$$\kappa\left({x^\prime\over \varepsilon}\right)=k^\prime \iff x^\prime\in Y^{\prime}_{k^{\prime}, \varepsilon}.$$
For a vectorial function $v=(v',v_3)$ and a scalar function $w$, we introduce the operators: $D_{\varepsilon}$, $\nabla_\varepsilon$ and ${\rm div}_{\varepsilon}$, by
\begin{equation*}
(D_{\varepsilon}v)_{i,j}=\partial_{x_j}v_i\text{\ for \ }i=1,2,3,\ j=1,2,\quad \quad (D_{\varepsilon}v)_{i,3}=\frac{1}{\varepsilon}\partial_{y_3}v_i\text{\ for \ }i=1,2,3,
\end{equation*}
\begin{equation*}
\nabla_{\varepsilon}w=(\nabla_{x^{\prime}}w,\frac{1}{\varepsilon}\partial_{y_3}w)^t, \quad \quad {\rm div}_{\varepsilon}v={\rm div}_{x^{\prime}}v^{\prime}+\frac{1}{\varepsilon}\partial_{y_3}v_3.
\end{equation*}

We denote by $|\mathcal{O}|$ the Lebesgue measure of $|\mathcal{O}|$ (3-dimensional if $\mathcal{O}$ is a 3-dimensional open set, 2-dimensional of $\mathcal{O}$ is a curve).

For every bounded set $\mathcal{O}$ and if $\varphi\in L^1(\mathcal{O})$, we define the average of $\varphi$ on $\mathcal{O}$ by
\begin{equation}\label{average}
\mathcal{M}_{\mathcal{O}}[\varphi]={1\over |\mathcal{O}|}\int_\mathcal{O}\varphi\,dx\,.
\end{equation}
Similarly, for every compact set $K$ of $Y$, if $\varphi\in L^1(\partial K)$ then 
$$\mathcal{M}_{\partial K}[\varphi]={1\over |{\partial K}|}\int_{\partial K}\varphi\,d\sigma\,,$$
is the average of $\varphi$ over $\partial K$.

We denote by $L^2_{\sharp} (Y)$, $H^1_{\sharp}(Y)$, the functional spaces
$$\begin{array}{l}\displaystyle
L^2_{\sharp} (Y)=\Big\{ v\in L^2_{loc}(Y)\, : \,\int_{Y}
|v|^2 dy <+\infty,\,  v(y'+k',y_3)=v(y)\hskip 0.2cm \forall k'\in\mathbb{Z}^2,\, \mbox{a.e. }y\in Y
\Big\},\end{array}
$$
and
$$\begin{array}{l}\displaystyle
H^1_{\sharp} (Y)= \Big\{ v \in H^1_{loc}(Y)\cap L^2_{\sharp} (Y):  \int_{Y}
|\nabla_{y}v|^2 dy <+\infty\Big\}.
\end{array}
$$

We denote by $:$ the full contraction of two matrices, i.e. for $A=(a_{i,j})_{1\leq i,j\leq 2}$ and $B=(b_{i,j})_{1\leq i,j\leq 2}$, we have $A:B=\sum_{i,j=1}^2 a_{ij}b_{ij}$.\\

Finally, we denote by $O_\varepsilon$ a generic real sequence, which tends to zero with $\varepsilon$ and can change from line to line, and by $C$ a generic positive constant which also can change from line to line.

\section{Main result}\label{S2}
In this section we describe the asymptotic behavior of a viscous fluid obeying (\ref{1}) in the geometry $\Omega_{\varepsilon}$ described in (\ref{Dominio1}). The proof of the corresponding results will be given in the next section.

\paragraph*{The variational formulation:} let us introduce the spaces $$H_{\varepsilon}=\left\{ \varphi\in H^1(\Omega_{\varepsilon})\ :\ \varphi=0 \text{ on } \partial \Lambda_\varepsilon\right\},\quad H^3_{\varepsilon}=\left\{ \varphi\in H^1(\Omega_{\varepsilon})^3\ :\ \varphi=0 \text{ on } \partial \Lambda_\varepsilon\right\},$$
and
$$\widetilde H_{\varepsilon}=\left\{ \tilde \varphi\in H^1(\widetilde \Omega_{\varepsilon})\ :\ \tilde \varphi=0 \text{ on } \partial  \Omega\right\},\quad \widetilde H^3_{\varepsilon}=\left\{ \tilde \varphi\in H^1(\widetilde \Omega_{\varepsilon})^3\ :\ \tilde \varphi=0 \text{ on } \partial  \Omega\right\}.$$

Then, the variational formulation of system (\ref{1}) is the following one:
\begin{equation}\label{fv1}
\left\{\begin{array}{l}
\displaystyle
\mu \int_{\Omega_\varepsilon}\!\!Du_\varepsilon:D \varphi\, dx-\int_{\Omega_\varepsilon}\!\!p_\varepsilon\,{\rm div}\,\varphi\, dx+\alpha \varepsilon^{\gamma}\int_{\partial S_\varepsilon}\!\!u_\varepsilon \cdot \varphi\, d\sigma(x) =\int_{\Omega_\varepsilon}\!\!f'_\varepsilon\cdot \varphi'\,dx+\int_{\partial S_\varepsilon}\!\!g'_\varepsilon \cdot \varphi'\, d\sigma(x),\ \forall \,\varphi\in H_\varepsilon^3,\\
\\ \displaystyle \int_{\Omega_\varepsilon} u_\varepsilon \cdot \nabla \psi\,dx=\int_{\partial  S_\varepsilon}( u_\varepsilon \cdot n) \psi\, d\sigma(x),\  \forall \, \psi\in H_\varepsilon.
\end{array}\right.
\end{equation}

Assume that $f_\varepsilon(x)=(f'_\varepsilon(x'),0)\in L^2( \omega)^3$ and $g_\varepsilon(x)= g(x'/\varepsilon)$, where $ g$ is a $Y'$-periodic function in $L^2(\partial  T)^2$. Under these assumptions, it is well known that (\ref{fv1}) has a unique solution $(u_{\varepsilon},p_{\varepsilon})\in H_\varepsilon^3\times L^2(\Omega_{\varepsilon})$  (see Theorem 4.1 and Remark 4.1 in \cite{Conca} for more details).\\

Our aim is to study the asymptotic behavior of $u_{\varepsilon}$ and $p_{\varepsilon}$ when $\varepsilon$ tends to zero. For this purpose, we use the dilatation in the variable $x_3$, i.e.
\begin{equation}\label{dilatacion}
y_3=\frac{x_3}{\varepsilon},
\end{equation}
in order to have the functions defined in the open set with fixed height $\widetilde \Omega_\varepsilon$ defined by (\ref{OmegaTilde}).

Namely, we define $\tilde{u}_{\varepsilon}\in \widetilde H_\varepsilon^3$, $\tilde{p}_{\varepsilon}\in L^2(\widetilde{\Omega}_{\varepsilon})$ by $$\tilde{u}_{\varepsilon}(x^{\prime},y_3)=u_{\varepsilon}(x^{\prime},\varepsilon y_3),\text{\ \ }\tilde{p}_{\varepsilon}(x^{\prime},y_3)=p_{\varepsilon}(x^{\prime},\varepsilon y_3), \text{\ \ } a.e.\text{\ } (x^{\prime},y_3)\in \widetilde{\Omega}_{\varepsilon}.$$
Using the transformation (\ref{dilatacion}), the system (\ref{1}) can be rewritten as
\begin{equation}\label{2}
\left\{
\begin{array}
[c]{r@{\;}c@{\;}ll}%
\displaystyle -\mu\Delta_{x^{\prime}} \tilde{u}_{\varepsilon}-\varepsilon^{-2}\mu\partial^2_{y_3}\tilde{u}_{\varepsilon}+ \nabla_{x^{\prime}} \tilde{p}_{\varepsilon}+\varepsilon^{-1}\partial_{y_3}\tilde{p}_{\varepsilon}e_3 &
= &
f_\varepsilon \text{\ in \ }\widetilde{\Omega}_{\varepsilon},\\
\noame
\displaystyle {\rm div}_{x^{\prime}}\, \tilde{u}^{\prime}_{\varepsilon}+ \varepsilon^{-1}\partial_{y_3}\tilde{u}_{\varepsilon,3} & = & 0 \text{\ in \ }\widetilde{\Omega}_{\varepsilon},\\
\noame
\tilde u_\varepsilon&=& 0 \text{\ on \ }\partial{\Omega},
\end{array}
\right. %
\end{equation}
with the non-homogeneous slip boundary condition, 
\begin{equation}\label{CF-2}
-\tilde p_\varepsilon\cdot n+\mu {\partial \tilde u_{\varepsilon}\over \partial n} + \alpha \varepsilon^\gamma \tilde u_\varepsilon= g_\varepsilon \text{\ on \ } \partial T_{\varepsilon}, 
\end{equation}
where $e_3=(0,0,1)^t$.

Taking in (\ref{fv1}) as test function $\tilde \varphi\left(x',x_3 / \varepsilon\right)$ with $\tilde \varphi \in \widetilde H_\varepsilon^3$ and $\tilde \psi\left(x',x_3/ \varepsilon\right)$ with $\tilde \psi \in \widetilde H_\varepsilon$, applying the change of variables (\ref{dilatacion}) and taking into account that $d\sigma(x)=\varepsilon d\sigma(x')dy_3$, the variational formulation of system (\ref{2})-(\ref{CF-2}) is then the following one:

\begin{equation}\label{fv2}
\left\{\begin{array}{l}
\displaystyle
\mu \int_{\widetilde \Omega_\varepsilon}D_\varepsilon \tilde u_\varepsilon:D_\varepsilon \tilde \varphi\, dx'dy_3-\int_{\widetilde \Omega_\varepsilon}\tilde p_\varepsilon\,{\rm div}_\varepsilon\,\tilde \varphi\, dx'dy_3+\alpha \varepsilon^{\gamma}\int_{\partial  T_\varepsilon}\tilde u_\varepsilon \cdot \tilde \varphi\, d\sigma(x')dy_3\\
\noame
\displaystyle \quad=\int_{\widetilde \Omega_\varepsilon}f'_\varepsilon\cdot \tilde \varphi'\,dx'dy_3+\int_{\partial  T_\varepsilon} g'_\varepsilon \cdot \tilde\varphi'\, d\sigma(x')dy_3,\quad \forall \,\tilde \varphi\in \widetilde H_\varepsilon^3,\\
\noame
\displaystyle
\int_{\widetilde \Omega_\varepsilon}\tilde u_\varepsilon\cdot \nabla_{\varepsilon}\tilde \psi\,dx'dy_3=\int_{\partial  T_\varepsilon}(\tilde u_\varepsilon \cdot n)\tilde \psi\, d\sigma(x')dy_3,\quad \forall \,\tilde \psi\in \widetilde H_\varepsilon.
\\
\end{array}\right.
\end{equation}

In the sequel, we assume that the data $f'_\varepsilon$ satisfies that there exists $f'\in L^2(\omega)^2$ such that
\begin{equation}\label{Hf}
\varepsilon f'_\varepsilon \rightharpoonup f' \quad \text{ weakly in }L^2(\omega)^2.
\end{equation}
Observe that, due to the periodicity of the obstacles, if $\displaystyle f'_\varepsilon={f'\over \varepsilon}$ where $f'\in L^2(\omega)^2$, then 
$$\chi_{\Omega_\varepsilon}f'_\varepsilon=\varepsilon f'_\varepsilon\rightharpoonup \theta f'\quad\hbox{in }L^2(\omega)^2,$$
assuming $\varepsilon f'_\varepsilon$ extended by zero outside of $\omega_\varepsilon$, where  $\theta$ denotes the proportion of the material in the cell $Y'$  given by 
$$\theta:={|Y'_f|\over |Y'|}.$$
We also define the constant $\displaystyle \mu_1:={|\partial T'|\over |Y'|}$.

\paragraph*{Main result:} our goal then is to describe the asymptotic behavior of this new sequence $(\tilde{u}_{\varepsilon}$, $\tilde{p}_{\varepsilon})$ when $\varepsilon$ tends to zero. The sequence of solutions $(\tilde{u}_{\varepsilon}$, $\tilde{p}_{\varepsilon})\in \widetilde H_{\varepsilon} \times  L^2(\widetilde{\Omega}_{\varepsilon})$ is not defined in a fixed domain independent of $\varepsilon$ but rather in a varying set $\widetilde{\Omega}_{\varepsilon}$. In order to pass the limit if $\varepsilon$ tends to zero, convergences in fixed Sobolev spaces (defined in $\Omega$) are used which requires first that $(\tilde{u}_{\varepsilon}$, $\tilde{p}_{\varepsilon})$ be extended to the whole domain $\Omega$. For the velocity,  we will denote by $\tilde U_\varepsilon$ the zero extension of $\tilde u_\varepsilon$ to the whole $\Omega$, and for the pressure we will denote  by $\tilde P_\varepsilon$ the zero extension of $\tilde p_\varepsilon$ to the whole $\Omega$.

Our main result referred to the asymptotic behavior of the solution of (\ref{2})-(\ref{CF-2}) is given by the following theorem.  
\begin{theorem}\label{MainTheorem} 
Let $(\tilde u_\varepsilon, \tilde p_\varepsilon)$ be the solution of (\ref{2})-(\ref{CF-2}). Then there exists an extension operator $\widetilde \varPi_\varepsilon\in \mathcal{L}(\widetilde H_\varepsilon^3;H_0^1(\Omega)^3)$ such that
\begin{itemize}
\item[i)]  if $\gamma<-1$, then 
 \begin{equation*}
\tilde \varPi_\varepsilon \tilde u_\varepsilon\rightharpoonup 0 \text{\ in \ }H^1_0(\Omega)^3.
\end{equation*}
Moreover, $(\varepsilon^{-1}\tilde U_\varepsilon, \varepsilon^{-\gamma}\tilde P_\varepsilon)$ is bounded in $H^1(0,1;L^2(\omega)^3)\times L^2(\omega)$ and any weak-limit point $(\tilde u(x^{\prime},y_3), \tilde p(x^{\prime}))$ of this sequence satisfies $\tilde u'=0$ on $y_3=\{0,1\}$, $\tilde u_3=0$ and the following Darcy type law:
\begin{equation}\label{effective1}
\left\{\begin{array}{l}
  \displaystyle \tilde v'(x')= -{\theta\over \alpha\,\mu_1}\nabla_{x'}\tilde p(x')\\
   \noame \displaystyle \tilde v_3(x')=0,
   \end{array}\right.
   \quad  \text{\ in \ } \omega\,,
\end{equation}
where $\tilde v(x')=\int_0^1\tilde u(x',y_3)dy_3$,
\item[ii) ] if $-1\leq \gamma<1$, then 
\begin{equation*}
\varepsilon^{{\gamma+1\over 2}}\tilde \varPi_\varepsilon \tilde u_\varepsilon\rightharpoonup 0\text{\ in \ }H^1_0(\Omega)^3.
\end{equation*}
Moreover, $(\varepsilon^{\gamma}\tilde U_\varepsilon, \varepsilon \tilde P_\varepsilon)$ is bounded in $H^1(0,1;L^2(\omega)^3)\times L^2(\omega)$ and any weak-limit point $(\tilde u(x^{\prime},y_3), \tilde p(x^{\prime}))$ of this sequence satisfies $\tilde u'=0$ on $y_3=\{0,1\}$,  $\tilde{u}_3=0$ and the following Darcy type law:
\begin{equation}\label{effective2}
\left\{\begin{array}{l}
  \displaystyle \tilde v'(x')= {\theta\over \alpha\,\mu_1}\left(f'-\nabla_{x'}\tilde p(x')+\mu_1\mathcal{M}_{\partial T'}[g']\right)\\
   \noame \displaystyle \tilde v_3(x')=0,
   \end{array}\right.
   \quad  \text{\ in \ } \omega\,,
\end{equation}
where $\tilde v(x')=\int_0^1\tilde u(x',y_3)dy_3$, 
\item[iii) ] if $\gamma \ge 1$, then 
\begin{equation*}
\varepsilon \tilde \varPi_\varepsilon \tilde u_\varepsilon\rightharpoonup 0\text{\ in \ }H^1_0(\Omega)^3.
\end{equation*}
Moreover, $(\tilde U_\varepsilon, \varepsilon^2 \tilde P_\varepsilon)$ is bounded in $H^1(0,1;L^2(\omega)^3)\times L^2(\omega)$ and any weak-limit point $(\tilde u(x^{\prime},y_3), \tilde p(x^{\prime}))$ of this sequence satisfies $\tilde u=0$ on $y_3=\{0,1\}$ and the following Darcy type law:
\begin{equation}\label{effective3}
 \left\{\begin{array}{l}
  \displaystyle\tilde v'(x')= -{\theta\over \mu}A\nabla_{x'}\tilde p(x')\\
   \noame \displaystyle \tilde v_3(x')=0,
     \end{array}\right.
   \quad  \text{\ in \ } \omega\,,
\end{equation}
where $\tilde v(x')=\int_0^1\tilde u(x',y_3)dy_3$, and the symmetric and positive permeability tensor $A\in \mathbb{R}^{2\times 2}$ is defined by its entries
$$A_{ij}={1\over |Y'_f|}\int_{Y'_f}Dw^i(y'):D w^j(y')\,dy,\quad i,j=1,2.$$
For $i=1,2$, $w^i(y')$, denote the unique solutions in $H^1_{\sharp}(Y'_f)^2$ of the local Stokes problems with homogeneous Neumann boundary conditions in 2D:
\begin{equation}\label{local_problem_3}\left\{\begin{array}{rl}
\displaystyle
-\Delta_{y'} w^i+\nabla_{y'} q^i=e_i&\quad\hbox{in }Y'_f\\
\noame\displaystyle
{\rm div}_{y'}\hat w^i=0&\quad\hbox{in }Y'_f\\
\noame\displaystyle
{\partial w^i\over \partial n}=0&\quad\hbox{on }\partial T',\\
\noame\displaystyle
w^i,q^i\quad Y'-{\rm periodic},\\
\noame\displaystyle
\mathcal{M}_{Y_f}[w^i]=0\,.
\end{array}\right.
\end{equation}
\end{itemize}
\end{theorem}

\begin{remark} We observe that in the homogenized problems related to system (\ref{2})-(\ref{CF-2}), the limit functions do not satisfy any incompressibility condition, so (\ref{effective1}), (\ref{effective2}) and (\ref{effective3}) do not identify in a unique way $(\tilde v, \tilde p)$. This is a consequence of the fact that the normal component of $\tilde u_\varepsilon$ does not vanish on the boundary of the cylinders, so the average fluid flow, given by $\tilde v$, is (eventually) represented by the motion of a compressible fluid. As pointed out in Conca \cite{Conca} (see Remark 2.1) and Cioranescu {\it et al.} \cite{Ciora0} (see Remarks 2.3 and 2.6), this result, which at first glance seems unexpected, can be justified as follows: the boundary condition (\ref{CF-2}) implies that the normal component $\tilde u_\varepsilon\cdot n$ of $\tilde u_\varepsilon$ is not necessarily zero on $\partial T_\varepsilon$ so we can not expect that the extension $\tilde U_\varepsilon$ will be a zero-divergence function.  In fact, from the second variational formulation in (\ref{fv2}), we have
$$\int_{\Omega} \tilde U_\varepsilon\cdot \nabla_\varepsilon \tilde\psi\,dx'dy_3=\int_{\partial  T_\varepsilon} (\tilde u_\varepsilon\cdot n)\tilde \psi\,d\sigma(x')dy_3,\quad \forall \psi \in \tilde H_\varepsilon,$$
and the term on the right-hand side is not necessarily zero. Therefore, by weak continuity, it is not possible to obtain an incompressibility condition of the form ${\rm div}_{x'}\tilde v'(x')=0$ in $\omega$ as it is obtained in the case with Dirichlet condition considered in \cite{Anguiano}.
Roughly speaking, $\tilde u_\varepsilon\cdot n\neq 0$ on $\partial T_\varepsilon$ means that some fluid ``dissapear" through the cylinders, and this fact implies that the incompressibility condition is not necessary verified at the limit in $\omega$. 
\end{remark}

\section{Proof of the main result}\label{S3}
In the context of homogenization of flow through porous media Arbogast {\it et al.} \cite{arbogast} use a $L^2$ dilation operator to resolve oscillations on a prescribed scale of weakly converging sequences. It was observed in Bourgeat {\it et al.} 
\cite{B-K-M} that this operator yields a characterization of two-scale convergence (see Allaire \cite{Allaire1} and Nguetseng \cite{Nghe}). Later, Cioranescu {\it et al.} \cite{Ciora,CioDonZak1} introduced the periodic unfolding method as a systematic approach to homogenization which can be used for $H^1$ functions and take into account the boundary of the holes by using the so-called boundary unfolding operator. In this section we prove our main result. In particular, Theorem \ref{MainTheorem}, is proved in Subsection \ref{s6} by means of a combination of the unfolding method applied to the horizontal variables, with a rescaling on the vertical variable. One of the main difficulties is to treat the surface integrals using an adaptation of the boundary unfolding operator. To apply this method, {\it a priori} estimates are established in Subsection \ref{s4} and some compactness results are proved in Subsection \ref{s5}.

\subsection{Some abstract results for thin domains and {\it a priori} estimates}\label{s4}
The {\it a priori} estimates independent of $\varepsilon$ for $\tilde u_\varepsilon$ and $\tilde p_\varepsilon$  will be obtained by using an adaptation of the unfolding method.  

\paragraph*{Some abstract results for thin domains:} let us introduce the adaption of the unfolding method in which we divide the domain $\widetilde \Omega_\varepsilon$ in cubes of lateral lengths $\varepsilon$ and vertical length $1$. For this purpose, given $\tilde{\varphi}\in L^p(\widetilde \Omega_\varepsilon)^3$, $1\leq p<+\infty$, (assuming $\tilde \varphi$ extended by zero outside of $\omega_\varepsilon$), we define $\hat{\varphi}_{\varepsilon}\in L^p(\mathbb{R}^2\times Y_f)^3$ by
\begin{equation}\label{vhat}
\hat{\varphi}_{\varepsilon}(x^{\prime},y)=\tilde{\varphi}\left( {\varepsilon}\kappa\left(\frac{x^{\prime}}{{\varepsilon}} \right)+{\varepsilon}y^{\prime},y_3 \right),\text{\ \ a.e. \ }(x^{\prime},y)\in \mathbb{R}^2 \times Y_f,
\end{equation}
where the function $\kappa$ is defined in (\ref{kappa_fun}).

\begin{remark}\label{remarkCV}
The restriction of $\hat{\varphi}_{\varepsilon}, $ to $Y^{\prime}_{f_k^{\prime}, {\varepsilon}}\times Y_f$ does not depend on $x^{\prime}$, whereas as a function of $y$ it is obtained from $\tilde{v}_{\varepsilon}, $ by using the change of variables 
\begin{equation}\label{CV}
y^{\prime}=\frac{x^{\prime}-{\varepsilon}k^{\prime}}{{\varepsilon}}, 
\end{equation}
which transforms $Y_{f_{k^{\prime}}, {\varepsilon}}$ into $Y_f$.
\end{remark}
\begin{proposition}\label{properties_unf} We have the following estimates:
\begin{itemize}
\item[i) ] for every $\tilde \varphi \in L^p(\widetilde \Omega_\varepsilon)$, $1\leq p +\infty$, we have
\begin{equation}\label{normv}\|\hat \varphi_\varepsilon\|_{L^p(\mathbb{R}^2\times Y_f)^3}= |Y'|^{1\over p}\|\tilde \varphi \|_{L^p(\widetilde \Omega_\varepsilon)^3},
\end{equation}
where $\hat\varphi_\varepsilon$ is given by (\ref{vhat}),
\item[ii) ] for every $\tilde \varphi \in W^{1,p}(\widetilde \Omega_\varepsilon)^3$, $1\leq p<+\infty$, we have that $\hat{\varphi}_\varepsilon$, given by (\ref{vhat}), belongs to $L^p(\mathbb{R}^2;W^{1,p}(Y_f)^3)$, and 
\begin{equation}\label{normDv}\|D_y \hat \varphi_\varepsilon\|_{L^p(\mathbb{R}^2\times Y_f)^{3\times 3}}=\varepsilon |Y'|^{1\over p}\|D_\varepsilon \tilde \varphi\|_{L^p(\widetilde \Omega_\varepsilon)^{3\times 3}}.
\end{equation}
\end{itemize}
\end{proposition}
\begin{proof}
Let us prove $i)$.  Using Remark \ref{remarkCV} and definition (\ref{vhat}), we have
\begin{eqnarray*}
\int_{\mathbb{R}^2\times Y_f}\left\vert \hat{\varphi}_{\varepsilon}(x^{\prime},y)\right\vert^pdx^{\prime}dy&=&\displaystyle\sum_{k^{\prime}\in \mathbb{Z}^2}\int_{Y^{\prime}_{k^{\prime},{\varepsilon}}}\int_{Y_f}\left\vert \hat{\varphi}_{\varepsilon}(x^{\prime},y)\right\vert^pdx^{\prime}dy\\
&=&\displaystyle\sum_{k^{\prime}\in \mathbb{Z}^2}\int_{Y^{\prime}_{k^{\prime},{\varepsilon}}}\int_{Y_f}\left\vert \tilde{\varphi}({\varepsilon}k^{\prime}+{\varepsilon}y^{\prime},y_3)\right\vert^pdx^{\prime}dy.
\end{eqnarray*}
We observe that $\tilde{\varphi}$ does not depend on $x^{\prime}$, then we can deduce
\begin{eqnarray*}
\int_{\mathbb{R}^2\times Y_f}\left\vert \hat{\varphi}_{\varepsilon}(x^{\prime},y)\right\vert^p dx^{\prime}dy= {\varepsilon}^2|Y'|\displaystyle\sum_{k^{\prime}\in \mathbb{Z}^2}\int_{Y_f}\left\vert \tilde{\varphi}({\varepsilon}k^{\prime}+{\varepsilon}y^{\prime},y_3)\right\vert^pdy.
\end{eqnarray*}
For every $k^{\prime}\in \mathbb{Z}^2$, by the change of variable (\ref{CV}), we have
\begin{equation}\label{mainCV_antiguo}
k^{\prime}+y^{\prime}=\frac{x^{\prime}}{{\varepsilon}}, \text{\ \ }dy^{\prime}=\frac{dx^{\prime}}{{\varepsilon}^2}\text{\ \ }\partial_{y^{\prime}}={\varepsilon}\partial_{x^{\prime}},
\end{equation}
and we obtain
$$
\begin{array}{rl}
\displaystyle\int_{\mathbb{R}^2\times Y_f}\left\vert\hat{\varphi}_{\varepsilon}(x^{\prime},y)\right\vert^pdx^{\prime}dy= &\displaystyle|Y'|\int_{\omega_\varepsilon\times (0,1)}\left\vert \tilde{\varphi}(x^{\prime},y_3)\right\vert^pdx^{\prime}dy_3 
\end{array}
$$
which gives (\ref{normv}).

Let us prove $ii)$. Taking into account the definition (\ref{vhat}) of $\hat{\varphi}_{\varepsilon}$ and observing that $\tilde{\varphi}$ does not depend on $x^{\prime}$, then we can deduce
\begin{eqnarray*}
\int_{\mathbb{R}^2\times Y_f}\left\vert D_{y^{\prime}}\hat{\varphi}_{\varepsilon}(x^{\prime},y)\right\vert^pdx^{\prime}dy= {\varepsilon}^2|Y'|\displaystyle\sum_{k^{\prime}\in \mathbb{Z}^2}\int_{Y_f}\left\vert D_{y^{\prime}}\tilde{\varphi}({\varepsilon}k^{\prime}+{\varepsilon}y^{\prime},y_3)\right\vert^pdy.
\end{eqnarray*}
By (\ref{mainCV_antiguo}), we obtain
\begin{eqnarray}
\int_{\mathbb{R}^2\times Y_f}\left\vert D_{y^{\prime}}\hat{\varphi}_{\varepsilon}(x^{\prime},y)\right\vert^pdx^{\prime}dy&=&{\varepsilon}^p|Y'|\displaystyle\sum_{k^{\prime}\in \mathbb{Z}^2}\int_{Y^{\prime}_{f_{k^{\prime}},{\varepsilon}}}\int_0^1\left\vert D_{x^{\prime}}\tilde{\varphi}(x^{\prime},y_3)\right\vert^pdx^{\prime}dy_3 \nonumber\\
&=& {\varepsilon}^p|Y'|\int_{\omega_\varepsilon\times (0,1)}\left\vert D_{x^{\prime}}\tilde{\varphi}(x^{\prime},y_3)\right\vert^pdx^{\prime}dy_3.
\label{estim_prop_1_prueba}
\end{eqnarray}
For the partial of the vertical variable, proceeding similarly to (\ref{normv}), we obtain
$$
\begin{array}{rl}
\displaystyle\int_{\mathbb{R}^2\times Y_f}\left\vert \partial_{y_3}\hat{\varphi}_{\varepsilon}(x^{\prime},y)\right\vert^pdx^{\prime}dy= &\displaystyle|Y'|\int_{\omega_\varepsilon\times (0,1)}\left\vert \partial_{y_3}\tilde{\varphi}(x^{\prime},y_3)\right\vert^pdx^{\prime}dy_3 \\
=&\displaystyle  \varepsilon^p |Y'|\int_{\omega_\varepsilon\times (0,1)}\left\vert{1\over \varepsilon} \partial_{y_3}\tilde{\varphi}(x^{\prime},y_3)\right\vert^pdx^{\prime}dy_3\,,
\end{array}
$$
which together with (\ref{estim_prop_1_prueba}) gives (\ref{normDv}).
\end{proof}

In a similar way, let us introduce the adaption of the unfolding method on the boundary of the obstacles $\partial   T_\varepsilon$ (see Cioranescu {\it et al.} \cite{CioDonZak1} for more details). For this purpose, given $\tilde{\varphi}\in L^p(\partial  T_\varepsilon)^3$, $1\leq p<+\infty$, we define $\hat{\varphi}^b_{\varepsilon}\in L^p(\mathbb{R}^2 \times\partial  T)^3$ by
\begin{equation}\label{vhat_b}
\hat{\varphi}^b_{\varepsilon}(x^{\prime},y)=\tilde{\varphi}\left( {\varepsilon}\kappa\left(\frac{x^{\prime}}{{\varepsilon}} \right)+{\varepsilon}y^{\prime},y_3 \right),\text{\ \ a.e. \ }(x^{\prime},y)\in \mathbb{R}^2 \times \partial  T,
\end{equation}
where the function $\kappa$ is defined in (\ref{kappa_fun}).

\begin{remark}\label{remark_unfmeth}
Observe that from this definition, if we consider $\tilde \varphi\in L^p(\partial T)$, $1\leq p <+\infty$, a $Y'$-periodic function, and we define $\tilde \varphi_\varepsilon(x',y_3)=\tilde \varphi(x'/ \varepsilon,y_3)$, it follows that $\hat \varphi_\varepsilon^b(x',y)=\tilde \varphi(y)$.\\
Observe that for $\tilde \varphi \in W^{1,p}(\widetilde \Omega_\varepsilon)^3$, $\hat{\varphi}^b_\varepsilon$ is the trace on $\partial T$ of $\hat{\varphi}_\varepsilon$. Therefore $\hat{\varphi}^b_\varepsilon$ has a similar properties as $\hat{\varphi}_\varepsilon$.
\end{remark}

We have the following property.

\begin{proposition}\label{properties_unf_b} If $\tilde \varphi \in L^p(\partial  T_\varepsilon)^3$, $1\leq p<+\infty$, then 
 \begin{equation}\label{unf_boun}
\|\hat \varphi^b_\varepsilon \|_{L^p(\mathbb{R}^2\times \partial T)^3}=\varepsilon^{1\over p} |Y'|^{1\over p}\|\tilde \varphi\|_{L^p(\partial T_\varepsilon)^3}\,,
\end{equation}
where $\hat \varphi^b_\varepsilon$ is given by (\ref{vhat_b}).
\end{proposition}
\begin{proof}
We take $(x',y_3)\in \partial T_{k',\varepsilon}$.  Taking into account (\ref{vhat_b}), we obtain

$$\begin{array}{rl}
\displaystyle
\int_{\mathbb{R}^2 \times \partial T} |\hat \varphi^b_\varepsilon (x',y)|^p\,dx'd\sigma(y)=&\displaystyle{\varepsilon}^2|Y'|\displaystyle\sum_{k^{\prime}\in \mathbb{Z}^2}\int_{\partial T} |\tilde{\varphi}({\varepsilon}k^{\prime}+{\varepsilon}y^{\prime},y_3)|^p\,d\sigma(y).
\end{array}$$
For every $k^{\prime}\in \mathbb{Z}^2$, by taking $x'=\varepsilon(k'+y')$, we have $d\sigma(x')=\varepsilon d\sigma(y')$. Since the thickness of the obstacles is one, we have that $d\sigma(x')dy_3=\varepsilon d\sigma(y)$. Hence
$$
\begin{array}{rl}
\displaystyle\int_{\mathbb{R}^2 \times \partial T} |\hat \varphi^b_\varepsilon (x',y)|^p\,dx'd\sigma(y)= &\displaystyle{\varepsilon}|Y'|\int_{ \partial T_\varepsilon} |\tilde{\varphi}(x',y_3)|^p\,d\sigma(x')dy_3,
\end{array}
$$
which gives (\ref{unf_boun}).
\end{proof}
Now, let us give two results which will be useful for obtaining {\it a priori} estimates of the solution $(\tilde u_\varepsilon,\tilde p_\varepsilon)$ of problem (\ref{2})-(\ref{CF-2}). These results are an extension of Cioranescu {\it et al.} (Proposition 5.3 and Corollary 5.4 in \cite{CioDonZak3}) to the thin domain case.

\begin{proposition}\label{Prop:prop_estim_1} Let $g\in L^2(\partial T')^3$ and $ \tilde \varphi\in H^1( \widetilde\Omega_\varepsilon)^3$, extended by zero in outside of $w_\varepsilon$. Let $\hat \varphi_\varepsilon$ be given by (\ref{vhat}). Then, there exists a positive constant $C$, independent of $\varepsilon$, such that 
\begin{equation}\label{prop_estim_1}
\left|\int_{\mathbb{R}^2\times \partial T} g(y')\cdot\hat \varphi_\varepsilon(x',y)\,dx'd\sigma(y)\right|\leq C\left|\mathcal{M}_{\partial T'}[g]\right|\left(\|\tilde \varphi\|_{L^2(\widetilde\Omega_\varepsilon)^3} +\varepsilon \|D_\varepsilon \tilde \varphi\|_{L^2(\widetilde \Omega_\varepsilon)^{3\times 3}}\right)\,.
\end{equation}
In particular, if $g=1$, there exists a positive constant $C$, independent of $\varepsilon$, such that
\begin{equation}\label{prop_estim_1_L1}
\left|\int_{\mathbb{R}^2\times \partial T} \hat \varphi_\varepsilon(x',y)\,dx'd\sigma(y)\right|\leq C\left(\|\tilde \varphi\|_{L^1(\widetilde\Omega_\varepsilon)^3} +\varepsilon \|D_\varepsilon \tilde \varphi\|_{L^1(\widetilde \Omega_\varepsilon)^{3\times 3}}\right)\,.
\end{equation}
\end{proposition}
\begin{proof} 
Due to density properties, it is enough to prove this estimate for functions in $\mathcal{D}(\mathbb{R}^3)^3$. Let $\tilde \varphi \in \mathcal{D}(\mathbb{R}^3)^3$, one has
$$\begin{array}{l}
\displaystyle
\left|\int_{\mathbb{R}^2\times \partial  T}
 g(y')\cdot\hat \varphi_\varepsilon(x',y)\,dx'd\sigma(y)\right|
\\
\noame
\displaystyle\qquad=\left|\int_{\mathbb{R}^2\times \partial  T} g(y')\cdot \tilde \varphi\left(\varepsilon \kappa\left({x'\over \varepsilon}\right)+ \varepsilon y',y_3\right)\,dx'd\sigma(y)
\right|
\\
\noame
\displaystyle\qquad\leq 
\left|\int_{\mathbb{R}^2\times \partial  T} g(y')\cdot \tilde \varphi\left(\varepsilon \kappa\left({x'\over \varepsilon}\right),y_3\right)\,dx'd\sigma(y)\right| \\
\noame\qquad\quad \displaystyle+\left|\int_{\mathbb{R}^2\times \partial  T} g(y')\cdot\left( \tilde \varphi\left(\varepsilon \kappa\left({x'\over \varepsilon}\right)+ \varepsilon y',y_3\right)-\tilde\varphi\left(\varepsilon \kappa\left({x'\over \varepsilon}\right),y_3\right)\right)\,dx'd\sigma(y)\right|
\\
\noame
\displaystyle\qquad\leq 
C\left|\mathcal{M}_{\partial T'}[g]\right|\left(\|\tilde \varphi\|_{L^2(\widetilde\Omega_\varepsilon)^3} +\varepsilon \|D_{x'} \tilde \varphi\|_{L^2(\widetilde \Omega_\varepsilon)^{3\times 3}}\right)
\\
\noame
\displaystyle\qquad\leq 
C\left|\mathcal{M}_{\partial T'}[g]\right|\left(\|\tilde \varphi\|_{L^2(\widetilde\Omega_\varepsilon)^3} +\varepsilon \|D_\varepsilon \tilde \varphi\|_{L^2(\widetilde \Omega_\varepsilon)^{3\times 3}}\right)\,,
\end{array}
$$
which implies (\ref{prop_estim_1}). In particular, if $g=1$, proceeding as above, we have
$$\begin{array}{l}
\displaystyle
\left|\int_{\mathbb{R}^2\times \partial  T}
\hat \varphi_\varepsilon(x',y)\,dx'd\sigma(y)\right|
\\
\noame
\displaystyle\qquad=\left|\int_{\mathbb{R}^2\times \partial  T}  \tilde \varphi\left(\varepsilon \kappa\left({x'\over \varepsilon}\right)+ \varepsilon y',y_3\right)\,dx'd\sigma(y)
\right|
\\
\noame
\displaystyle\qquad\leq 
\left|\int_{\mathbb{R}^2\times \partial  T}  \tilde \varphi\left(\varepsilon \kappa\left({x'\over \varepsilon}\right),y_3\right)\,dx'd\sigma(y)\right| \\
\noame\qquad\quad \displaystyle+\left|\int_{\mathbb{R}^2\times \partial  T} \left( \tilde \varphi\left(\varepsilon \kappa\left({x'\over \varepsilon}\right)+ \varepsilon y',y_3\right)-\tilde\varphi\left(\varepsilon \kappa\left({x'\over \varepsilon}\right),y_3\right)\right)\,dx'd\sigma(y)\right|
\\
\noame
\displaystyle\qquad\leq 
C\left(\|\tilde \varphi\|_{L^1(\widetilde\Omega_\varepsilon)^3} +\varepsilon \|D_{x'} \tilde \varphi\|_{L^1(\widetilde \Omega_\varepsilon)^{3\times 3}}\right)
\\
\noame
\displaystyle\qquad\leq 
C\left(\|\tilde \varphi\|_{L^1(\widetilde\Omega_\varepsilon)^3} +\varepsilon \|D_\varepsilon \tilde \varphi\|_{L^1(\widetilde \Omega_\varepsilon)^{3\times 3}}\right)\,,
\end{array}
$$
which implies (\ref{prop_estim_1_L1}).

\end{proof}
\begin{corollary} Let $g\in L^2(\partial  T)^3$ be a $Y'$-periodic  function. Then, for every $\tilde \varphi\in H^1(\widetilde \Omega_\varepsilon)^3$, we have
that there exists a positive constant $C$, independent of $\varepsilon$, such that
\begin{equation}\label{gvestimate}
\left|\int_{\partial T_\varepsilon} g(x'/\varepsilon)\cdot \tilde \varphi(x',y_3) d\sigma(x')dy_3\right|\leq {C\over \varepsilon}\left(\|\tilde \varphi\|_{L^2(\widetilde\Omega_\varepsilon)^3} + \varepsilon \|D_\varepsilon \tilde \varphi\|_{L^2(\widetilde \Omega_\varepsilon)^{3\times 3}}\right)\,.\end{equation}
In particular, if $g=1$, there exists a positive constant $C$, independent of $\varepsilon$, such that
\begin{equation}\label{gvestimate_L1}
\left|\int_{\partial T_\varepsilon}\tilde \varphi(x',y_3) d\sigma(x')dy_3\right|\leq {C\over \varepsilon}\left(\|\tilde \varphi\|_{L^1(\widetilde\Omega_\varepsilon)^3} + \varepsilon \|D_\varepsilon \tilde \varphi\|_{L^1(\widetilde \Omega_\varepsilon)^{3\times 3}}\right)\,.
\end{equation}
\end{corollary}

\begin{proof} 
Since $\tilde \varphi\in H^1(\widetilde \Omega_\varepsilon)^3$, then $\hat{\varphi}^b_\varepsilon$ has similar properties as $\hat{\varphi}_\varepsilon$. By using Proposition \ref{properties_unf_b} with $p=1$ and Remark \ref{remark_unfmeth}, we have 
$$
\left|\int_{\partial T_\varepsilon}g(x'/\varepsilon)\cdot \tilde \varphi(x',y_3) d\sigma(x')dy_3\right|
=
{1\over \varepsilon |Y'|} \left|\int_{\mathbb{R}^2\times \partial T}g(y')\cdot \hat \varphi_\varepsilon(x',y)\, dx'd\sigma(y)\right|,
$$
and by Proposition \ref{Prop:prop_estim_1}, we can deduce estimates (\ref{gvestimate}) and (\ref{gvestimate_L1}).

\end{proof}
Moreover, for the proof of the {\it a priori} estimates for the velocity, we need the following lemma due to Conca \cite{Conca} generalized to a thin domain $\Omega_\varepsilon$.

\begin{lemma}\label{Conca}
There exists a constant $C$ independent of $\varepsilon$, such that, for any function $\varphi\in H_{\varepsilon}$, one has 
\begin{equation}\label{conca_equation}
\left\Vert \varphi\right\Vert_{L^2(\Omega_{\varepsilon})^3}\leq C\left(\varepsilon\left\Vert D\varphi\right\Vert_{L^2(\Omega_{\varepsilon})^{{3\times3}}}+\varepsilon^{1\over 2}\|\varphi \|_{L^2(\partial S_\varepsilon)^3}\right).
\end{equation}
\end{lemma}
\begin{proof}
We observe that the microscale of the porous medium $\varepsilon$ is similar than the thickness of the domain $\varepsilon$, which lead us to divide the domain $\Omega_\varepsilon$ in small cubes of lateral length $\varepsilon$ and vertical length $\varepsilon$. We consider the periodic cell $Y$. The function $\varphi \to \left(\|D \varphi \|^2_{L^2(Y_f)^{3\times 3}}+\|\varphi \|^2_{L^2(\partial T)^3} \right)^{1/2}$ is a norm in $H^1(Y_f)^3$, equivalent to the $H^1(Y_f)^3$-norm (see Ne${\rm \check c}$as \cite{Necas}). Therefore, for any function $\varphi(z)\in H^1(Y_f)^3$, we have
\begin{equation}\label{p1}
\int_{Y_f}\left\vert \varphi \right\vert ^2dz\leq C\left(\int_{Y_f}\left\vert D_{z}\varphi\right\vert ^2dz+\int_{\partial T}\left\vert \varphi\right\vert ^2d\sigma(z)\right),
\end{equation}
where the constant $C$ depends only on $Y_f$.

Then, for every $k^{\prime}\in \mathbb{Z}^2$, by the change of variable 
\begin{equation}\label{mainCV}
k^{\prime}+z^{\prime}=\frac{x^{\prime}}{{\varepsilon}}, \text{\ \ }z_3=\frac{x_3}{\varepsilon}, \text{\ \ }dz=\frac{dx}{\varepsilon^3},\text{\ \ }\partial_{z}={\varepsilon}\partial_{x}, \text{\ \ }d\sigma(x)=\varepsilon^2\,d\sigma(z),
\end{equation}
we rescale (\ref{p1}) from $Y_f$ to $Q_{f_{k^{\prime}},{\varepsilon}}=Y^{\prime}_{f_{k^{\prime}},{\varepsilon}}\times (0,\varepsilon)$. This yields that, for any function $\varphi(x)\in H^1(Q_{f_{k^{\prime}},{\varepsilon}})^3$, one has
\begin{eqnarray}\label{p2}
\int_{Q_{f_{k^{\prime}},{\varepsilon}}}\left\vert \varphi\right\vert ^2dx&\leq& C\left({\varepsilon}^2\int_{Q_{f_{k^{\prime}},{\varepsilon}}}\left\vert D_{x}\varphi\right\vert ^2dx+\varepsilon\int_{T^{\prime}_{k^{\prime},{\varepsilon}}\times (0,\varepsilon)}\left\vert \varphi\right\vert ^2d\sigma(x)\right),
\end{eqnarray}
with the same constant $C$ as in (\ref{p1}). Summing the inequality (\ref{p2}) for all the periods $Q_{f_{k^{\prime}},{\varepsilon}}$ and $T^{\prime}_{k^{\prime},{\varepsilon}}\times (0,\varepsilon)$, gives 
\begin{eqnarray*}
\int_{\Omega_\varepsilon}\left\vert \varphi\right\vert ^2dx&\leq& C\left({\varepsilon}^2\int_{\Omega_\varepsilon}\left\vert D_{x}\varphi\right\vert ^2dx+\varepsilon\int_{\partial S_\varepsilon}\left\vert \varphi\right\vert ^2d\sigma(x)\right).
\end{eqnarray*}
In fact, we must consider separately the periods containing a portion of $\partial \omega$, but they yield at a distance $O(\varepsilon)$ of $\partial \omega$, where $\varphi$ is zero. Therefore, using Poincare's inequality one can easily verify that in this part (\ref{conca_equation}) holds without considering the boundary term occuring in (\ref{conca_equation}).

\end{proof}
Considering the change of variables given in (\ref{dilatacion}) and taking into account that $d\sigma(x)=\varepsilon d\sigma(x')dy_3$, we obtain the following result for the domain $\widetilde \Omega_\varepsilon$.
\begin{corollary}\label{Poincare}
There exists a constant $C$ independent of $\varepsilon$, such that, for any function $\tilde \varphi\in \widetilde H_\varepsilon^3$, one has 
\begin{equation}\label{conca_equation1}
\left\Vert \tilde \varphi\right\Vert_{L^2(\widetilde \Omega_{\varepsilon})^3}\leq C\left(\varepsilon\left\Vert D_\varepsilon \tilde \varphi\right\Vert_{L^2(\widetilde \Omega_{\varepsilon})^{{3\times3}}}+\varepsilon^{1\over 2}\|\tilde \varphi \|_{L^2(\partial T_\varepsilon)^3}\right).
\end{equation}
\end{corollary}
The presence in (\ref{1}) of the stress tensor in the boundary condition implies that the extension of the velocity is no longer obvious. If we consider the Stokes system with Dirichlet boundary condition on the obstacles, the velocity would be extended by zero in the obstacles. However, in this case, we need another kind of extension for the case in which the velocity is non-zero on the obstacles.   Since in the extension required, the vertical variable is not concerned, therefore the proof of the required statement is basically the extension of the result given in Cioranescu and Saint-Jean Paulin \cite{CioSaint1,CioSaint2} to the time-depending case given in Cioranescu and Donato \cite{C-D-exact}, so we omit the proof. We remark that the extension is not divergence free, so we cannot expect the homogenized solution to be divergence free. Hence we cannot use test functions that are divergence free in the variational formulation, which implies that the pressure has to be included.
\begin{lemma}\label{lemma_extension}
There exists an extension operator $\varPi_\varepsilon\in \mathcal{L}(H_\varepsilon^3;H_0^1(\Lambda_\varepsilon)^3)$ and a positive constant $C$, independent of $\varepsilon$, such that 
\begin{equation*}
\varPi_\varepsilon \varphi(x)=\varphi (x), \quad \text{ if \ } x\in \Omega_\varepsilon,
\end{equation*}
\begin{equation*}
\|D\varPi_\varepsilon \varphi \|_{L^2(\Lambda_\varepsilon)^{3 \times 3}}\leq C\|D \varphi \|_{L^2(\Omega_\varepsilon)^{3 \times 3}}, \quad \forall\, \varphi \in H_\varepsilon^3.
\end{equation*}
\end{lemma}
Considering the change of variables given in (\ref{dilatacion}), we obtain the following result for the domain $\widetilde \Omega_\varepsilon$.
\begin{corollary}\label{corollary_para_Poincare}
There exists an extension operator $\widetilde \varPi_\varepsilon\in \mathcal{L}(\widetilde H_\varepsilon^3;H_0^1(\Omega)^3)$ and a positive constant $C$, independent of $\varepsilon$, such that 
\begin{equation*}
\widetilde \varPi_\varepsilon \tilde \varphi(x',y_3)=\tilde \varphi (x',y_3), \quad \text{ if \ } (x',y_3)\in \widetilde \Omega_\varepsilon,
\end{equation*}
\begin{equation*}
\|D_\varepsilon \widetilde \varPi_\varepsilon \tilde \varphi \|_{L^2(\Omega)^{3 \times 3}}\leq C\|D_\varepsilon  \tilde \varphi \|_{L^2(\widetilde \Omega_\varepsilon)^{3 \times 3}}, \quad \forall\, \tilde \varphi \in \widetilde H_\varepsilon^3.
\end{equation*}
\end{corollary}
Using Corollary \ref{corollary_para_Poincare}, we obtain a Poincar\'e inequality in $\widetilde H_\varepsilon^3$.
\begin{corollary}
There exists a constant $C$ independent of $\varepsilon$, such that, for any function $\tilde \varphi\in \widetilde H_{\varepsilon}^3$, one has 
\begin{equation}\label{Poincare_delgada}
\left\Vert \tilde \varphi\right\Vert_{L^2(\widetilde \Omega_{\varepsilon})^3}\leq C\left\Vert D_\varepsilon \tilde \varphi\right\Vert_{L^2(\widetilde \Omega_{\varepsilon})^{{3\times3}}}.
\end{equation}
\end{corollary}
\begin{proof}
We observe that 
\begin{equation}\label{4.9_1}
\|\tilde \varphi \|_{L^2(\widetilde \Omega_\varepsilon)^3}\leq \|\widetilde \varPi_\varepsilon \tilde \varphi \|_{L^2( \Omega)^3},\quad \forall\, \tilde \varphi \in \widetilde H_\varepsilon^3.
\end{equation}
Since $\widetilde \varPi_\varepsilon \tilde \varphi \in H_0^1(\Omega)^3$, we can apply the Poincar\'e inequality in $H_0^1(\Omega)$ and then taking into account Corollary \ref{corollary_para_Poincare}, we get
\begin{equation*}
\|\widetilde \varPi_\varepsilon \tilde \varphi \|_{L^2( \Omega)^3}\leq C\|D \widetilde \varPi_\varepsilon \tilde \varphi \|_{L^2( \Omega)^{3\times 3}}\leq C\|D_\varepsilon \widetilde \varPi_\varepsilon \tilde \varphi \|_{L^2( \Omega)^{3\times 3}}\leq C\|D_\varepsilon  \tilde \varphi \|_{L^2(\widetilde \Omega_\varepsilon)^{3\times 3}}.
\end{equation*}
This together with (\ref{4.9_1}) gives (\ref{Poincare_delgada}).
\end{proof}

Now, for the proof of the {\it a priori} estimates for the pressure, we also need the following lemma due to Conca \cite{Conca} generalized to a thin domain $\Omega_\varepsilon$.
\begin{lemma}\label{Conca_pressure}
There exists a constant $C$ independent of $\varepsilon$, such that, for each $q\in L^2(\Omega_{\varepsilon})$, there exists $\varphi=\varphi(q)\in H_{\varepsilon}$, such that
\begin{equation}\label{conca_pressure1}
{\rm div}\, \varphi=q \ \text{in } \Omega_\varepsilon,
\end{equation}
\begin{equation}\label{conca_pressure2}
\|\varphi \|_{L^2(\Omega_\varepsilon)^3}\leq C\, \| q \|_{L^2(\Omega_\varepsilon)},\quad \|D \varphi \|_{L^2(\Omega_\varepsilon)^{3\times 3}}\leq {C \over \varepsilon} \| q \|_{L^2(\Omega_\varepsilon)}.
\end{equation}
\end{lemma}
\begin{proof}
Let $q \in L^2(\Omega_\varepsilon)$ be given. We extend $q$ inside the cylinders by means of the function:
$$Q(x)=\left\{
\begin{array}{rcl}
q(x) &\text{if}& x\in \Omega_\varepsilon\\
\noame 
\displaystyle{-1 \over |\Lambda_\varepsilon-\Omega_\varepsilon|}\int_{\Omega_\varepsilon}q(x) \,dx &\text{if}& x\in \Lambda_\varepsilon-\Omega_\varepsilon.
\end{array}\right.
$$
It is follows that $Q \in L^2_0(\Lambda_\varepsilon)=\{p\in L^2(\Lambda_\varepsilon)\,:\, \int_{\Lambda_\varepsilon}p\,dx=0\}$ and
\begin{equation}\label{5.9b}
\|Q\|_{L^2(\Lambda_\varepsilon)}^2=\|q\|_{L^2(\Omega_\varepsilon)}^2+{1\over |\Lambda_\varepsilon-\Omega_\varepsilon|}\left(\int_{\Omega_\varepsilon}q(x)\, dx\right)^2.
\end{equation}
Since $|\Lambda_\varepsilon-\Omega_\varepsilon|$ is bounded from below by a positive number, it follows from (\ref{5.9b}) and Cauchy-Schwartz inequality that
\begin{equation}\label{5.10}
\|Q\|_{L^2(\Lambda_\varepsilon)}\leq C\|q\|_{L^2(\Omega_\varepsilon)}.
\end{equation}
Besides that, since $Q \in L^2_0(\Lambda_\varepsilon)$, it follows from Maru${\rm \check s}$i\'c and Maru${\rm \check s}$i\'c-Paloka (Lemma 20 in \cite{Marusic}) that we can find $\varphi \in H_0^1(\Lambda_\varepsilon)^3$ such that
\begin{equation}\label{5.11a}
{\rm div}\, \varphi=Q \ \text{\ in \ }\Lambda_\varepsilon,
\end{equation}
\begin{equation}\label{5.11b}
\|  \varphi \|_{L^2(\Lambda_\varepsilon)^3}\leq C\, \| Q \|_{L^2(\Lambda_\varepsilon)},\quad \| D \varphi \|_{L^2(\Lambda_\varepsilon)^{3\times 3}}\leq {C \over \varepsilon}\| Q \|_{L^2(\Lambda_\varepsilon)}.
\end{equation}
Let us consider $\varphi_{|\Omega_\varepsilon}$: it belongs to $H_{\varepsilon}$. Moreover, (\ref{conca_pressure1}) follows from (\ref{5.11a}) and the estimates (\ref{conca_pressure2}) follows from (\ref{5.11b}) and (\ref{5.10}). 

\end{proof}
Considering the change of variables given in (\ref{dilatacion}), we obtain the following result for the domain $\widetilde \Omega_\varepsilon$.
\begin{corollary}\label{Conca_pressure2}
There exists a constant $C$ independent of $\varepsilon$, such that, for each $\tilde q\in L^2(\widetilde\Omega_{\varepsilon})$, there exists $\tilde \varphi=\tilde \varphi(\tilde q)\in \widetilde H_{\varepsilon}$, such that
\begin{equation*}\label{conca_pressure1tilde}
{\rm div}_{\varepsilon}\, \tilde \varphi=\tilde q \ \text{in } \widetilde \Omega_\varepsilon,
\end{equation*}
\begin{equation*}\label{conca_pressure2tilde}
\|\tilde \varphi \|_{L^2(\widetilde \Omega_\varepsilon)^3}\leq C\, \| \tilde q \|_{L^2(\widetilde \Omega_\varepsilon)},\quad \|D_\varepsilon \tilde \varphi \|_{L^2(\widetilde \Omega_\varepsilon)^{3\times 3}}\leq {C \over \varepsilon} \| \tilde q \|_{L^2(\widetilde\Omega_\varepsilon)}.
\end{equation*}
\end{corollary}

\paragraph*{A priori estimates for ($\tilde{u}_{\varepsilon}$, $\tilde{p}_{\varepsilon}$) in $\widetilde \Omega_\varepsilon$:} first, let us obtain some {\it a priori} estimates for $\tilde{u}_{\varepsilon}$ for different values of $\gamma$.
\begin{lemma}\label{Lemma_a3} We distinguish three cases:
\begin{itemize}
\item[i) ] for $\gamma<-1$, then there exists a constant $C$ independent of $\varepsilon$, such that
\begin{equation}\label{a3}
\|\tilde u_\varepsilon\|_{L^2(\widetilde \Omega_\varepsilon)^3}\leq C\varepsilon,\quad \|D_\varepsilon \tilde u_\varepsilon\|_{L^2(\widetilde \Omega_\varepsilon)^{3\times 3}}\leq C\,.
\end{equation}
\item[ii) ] for $-1\leq\gamma<1$, then there exists a constant $C$ independent of $\varepsilon$, such that
\begin{equation}\label{a3-2}
\| \tilde u_\varepsilon\|_{L^2(\widetilde \Omega_\varepsilon)^3}\leq C\varepsilon^{-\gamma},\quad \|D_\varepsilon \tilde u_\varepsilon\|_{L^2(\widetilde \Omega_\varepsilon)^{3\times 3}}\leq C\varepsilon^{-{1+\gamma\over 2}}\,.
\end{equation}
\item[iii) ] for $\gamma\geq 1$, then there exists a constant $C$ independent of $\varepsilon$, such that
\begin{equation}\label{a3-3}
\| \tilde u_\varepsilon\|_{L^2(\widetilde \Omega_\varepsilon)^3}\leq C\varepsilon^{-1},\quad \|D_\varepsilon \tilde u_\varepsilon\|_{L^2(\widetilde \Omega_\varepsilon)^{3\times 3}}\leq C\varepsilon^{-1}\,.
\end{equation}
\end{itemize}
\end{lemma}
\begin{proof}
Taking $\tilde{u}_{\varepsilon}\in \widetilde H_\varepsilon^3$ as function test in (\ref{fv2}), we have 
\begin{equation}\label{a1}
\mu\left\Vert D_{\varepsilon}\tilde{u}_{\varepsilon}\right\Vert_{L^2(\widetilde{\Omega}_{\varepsilon})^{{3\times3}}}^2+\alpha\varepsilon^\gamma\|\tilde u_\varepsilon\|^2_{L^2(\partial  T_\varepsilon)^3}
=\int_{\widetilde{\Omega}_{\varepsilon}}f'_\varepsilon\cdot\tilde{u}'_{\varepsilon}\,dx^{\prime}dy_3+ \int_{\partial T_\varepsilon} g'_\varepsilon \cdot \tilde u'_\varepsilon\, d\sigma(x')dy_3\,.
\end{equation}
Using Cauchy-Schwarz's inequality and $f'_\varepsilon \in L^2(\Omega)^2$, we obtain that 
\begin{eqnarray*}
\int_{\widetilde{\Omega}_{\varepsilon}}f'_\varepsilon\cdot\tilde{u}'_{\varepsilon}\,dx^{\prime}dy_3\leq C \left\Vert\tilde{u}_{\varepsilon} \right\Vert_{L^2(\widetilde{\Omega}_{\varepsilon})^3},
\end{eqnarray*}
and by using that $ g' \in L^2(\partial T)^2$ is a $Y'$-periodic function and estimate (\ref{gvestimate}), we have
$$
\left|\int_{\partial T_\varepsilon} g'_\varepsilon\,\cdot \tilde u'_\varepsilon\, d\sigma(x')dy_3\right|\leq {C\over \varepsilon}\left(\|\tilde u_\varepsilon\|_{L^2(\widetilde\Omega_\varepsilon)^3} + \varepsilon \|D_\varepsilon \tilde u_\varepsilon\|_{L^2(\widetilde \Omega_\varepsilon)^{3\times 3}}\right)\,.
$$
Putting these estimates in (\ref{a1}), we get
\begin{equation}\label{a2}
\mu\left\Vert D_{\varepsilon}\tilde{u}_{\varepsilon}\right\Vert_{L^2(\widetilde{\Omega}_{\varepsilon})^{{3\times3}}}^2 + \alpha\varepsilon^\gamma\|\tilde u_\varepsilon\|^2_{L^2(\partial  T_\varepsilon)^3}
\leq C\left(\|D_\varepsilon \tilde u_\varepsilon\|_{L^2(\widetilde \Omega_\varepsilon)^{3\times 3}} +  \varepsilon^{-1} \|\tilde u_\varepsilon\|_{L^2(\widetilde\Omega_\varepsilon)^3}\right)\,.
\end{equation}
In particular, if we use the Poincar\'e inequality (\ref{Poincare_delgada}) in (\ref{a2}), we have
\begin{equation}\label{general1}
\left\Vert D_{\varepsilon}\tilde{u}_{\varepsilon}\right\Vert_{L^2(\widetilde{\Omega}_{\varepsilon})^{{3\times3}}} \leq {C \over \varepsilon},
\end{equation}
therefore (independently of $\gamma\in \mathbb{R}$), using again (\ref{Poincare_delgada}), we get
\begin{equation}\label{general2}
\left\Vert \tilde{u}_{\varepsilon}\right\Vert_{L^2(\widetilde{\Omega}_{\varepsilon})^{3}} \leq {C \over \varepsilon}.
\end{equation}
These estimates can be refined following the different values of $\gamma$. To do so, observe that from estimate (\ref{conca_equation1}) we have
$$\varepsilon^{-1}\|\tilde u_\varepsilon\|_{L^2(\widetilde \Omega_\varepsilon)^3}\leq C\left(\|D_\varepsilon \tilde u_\varepsilon\|_{L^2(\widetilde \Omega_\varepsilon)^{3\times 3}}+\varepsilon^{-{1\over 2}}\|\tilde u_\varepsilon\|_{L^2(\partial  T_\varepsilon)^3}\right)\,.$$
Using Young's inequality, we get
\begin{equation*}\label{PoinYoung-1}
\varepsilon^{-{1\over 2}}\|\tilde u_\varepsilon\|_{L^2(\partial  T_\varepsilon)^3}\leq \varepsilon^{-{1+\gamma\over 2}}\varepsilon^{\gamma\over 2}\|\tilde u_\varepsilon\|_{L^2(\partial  T_\varepsilon)^3}\leq 
{2\over \alpha}\varepsilon^{-1-\gamma}+ {\alpha\over 2}\varepsilon^\gamma\|\tilde u_\varepsilon\|^2_{L^2(\partial  T_\varepsilon)^3}.
\end{equation*}
Consequently, from (\ref{a2}), we get
\begin{equation*}\label{a2-3}
\mu\left\Vert D_{\varepsilon}\tilde{u}_{\varepsilon}\right\Vert_{L^2(\widetilde{\Omega}_{\varepsilon})^{{3\times3}}}^2 + {\alpha\over 2}\varepsilon^\gamma\|\tilde u_\varepsilon\|^2_{L^2(\partial T_\varepsilon)^3}
\leq C\left(\|D_\varepsilon \tilde u_\varepsilon\|_{L^2(\widetilde \Omega_\varepsilon)^{3\times 3}} +\varepsilon^{-1-\gamma}\right)\,,
\end{equation*}
which applying in a suitable way the Young inequality gives
\begin{equation}\label{a2-4}
\mu\left\Vert D_{\varepsilon}\tilde{u}_{\varepsilon}\right\Vert_{L^2(\widetilde{\Omega}_{\varepsilon})^{{3\times3}}}^2 + \alpha\varepsilon^\gamma\|\tilde u_\varepsilon\|^2_{L^2(\partial  T_\varepsilon)^3}
\leq C\left(1+\varepsilon^{-1-\gamma}\right).
\end{equation}
For the case when $\gamma< -1$, estimate (\ref{a2-4}) reads
$$\|D_\varepsilon \tilde u_\varepsilon\|_{L^2(\widetilde\Omega_\varepsilon)^{3\times 3}}\leq C,\quad\|\tilde u_\varepsilon\|_{L^2(\partial  T_\varepsilon)^3}\leq C\varepsilon^{-{\gamma\over 2}}.$$
Then, estimate (\ref{conca_equation1}) gives
$$\|\tilde u_\varepsilon\|_{L^2(\widetilde\Omega_\varepsilon)^3}\leq C(\varepsilon+\varepsilon^{1-\gamma\over 2})\leq C\varepsilon,$$
since $1\leq (1-\gamma)/2$, and so, we have proved  (\ref{a3}).\\

For $\gamma\ge-1$, estimate (\ref{a2-4}) reads
$$\|D_\varepsilon \tilde u_\varepsilon\|_{L^2(\widetilde\Omega_\varepsilon)^{3\times 3}}\leq C\varepsilon^{-{1+\gamma\over 2}}, \quad \|\tilde u_\varepsilon\|_{L^2(\partial T_\varepsilon)^3}\leq C\varepsilon^{-{1\over 2}-\gamma}\,.$$
Applying estimate (\ref{conca_equation1}), we get
$$\|\tilde u_\varepsilon\|_{L^2(\widetilde\Omega_\varepsilon)^3}\leq  C(\varepsilon^{1-\gamma\over 2}+\varepsilon^{-\gamma})\leq C\varepsilon^{-\gamma}\,$$
since $-\gamma\leq(1-\gamma)/2$. Then, we have proved (\ref{a3-2}) for $-1\leq \gamma <1$. Observe that for $\gamma \geq 1$, the estimates (\ref{general1})-(\ref{general2}) are the optimal ones, so we have (\ref{a3-3}).

\end{proof}
We will prove now {\it a priori} estimates for the pressure $\tilde p_\varepsilon$ for different values of $\gamma$.

\begin{lemma}\label{Lemma_a3_P} We distinguish three cases: 
\begin{itemize}
\item[i)] for $\gamma<-1$, then there exists a constant $C$ independent of $\varepsilon$, such that
\begin{equation}\label{Estim_pep_1}
\|\tilde p_\varepsilon\|_{L^2(\widetilde \Omega_\varepsilon)}\leq C\,\varepsilon^{\gamma}.
\end{equation}
\item[ii)] for $-1\leq \gamma<1$, then there exists a constant $C$ independent of $\varepsilon$, such that
\begin{equation}\label{Estim_pep_2}
\| \tilde p_\varepsilon\|_{L^2(\widetilde\Omega_\varepsilon)}\leq C\,\varepsilon^{-1}.
\end{equation}
\item[iii)] for $\gamma\geq 1$, then there exists a constant $C$ independent of $\varepsilon$, such that
\begin{equation}\label{Estim_pep_3}
\| \tilde p_\varepsilon\|_{L^2(\widetilde \Omega_\varepsilon)}\leq C\,\varepsilon^{-2}.
\end{equation}
\end{itemize}
\end{lemma}

\begin{proof}
Let $\tilde \varPhi\in L^2(\widetilde \Omega_\varepsilon)$. From Corollary \ref{Conca_pressure2}, there exists $\tilde \varphi\in \widetilde H_\varepsilon^3$ such that 
\begin{equation}\label{5.10_new}
{\rm div}_{\varepsilon}\, \tilde \varphi=\tilde \varPhi \ \text{in } \widetilde \Omega_\varepsilon, \quad \| \tilde \varphi \|_{L^2(\widetilde \Omega_\varepsilon)^3}\leq C\, \| \tilde \varPhi \|_{L^2(\widetilde\Omega_\varepsilon)}, \quad
\|D_\varepsilon \tilde \varphi \|_{L^2(\widetilde \Omega_\varepsilon)^{3\times 3}}\leq {C \over \varepsilon} \| \tilde \varPhi \|_{L^2(\widetilde\Omega_\varepsilon)}.
\end{equation}
Taking $\tilde \varphi\in \widetilde H_\varepsilon^3$ as function test in (\ref{fv2}), we have
\begin{equation}\label{5.11_new}
\begin{array}{rl}
\displaystyle
\left|\int_{\widetilde \Omega_\varepsilon}\tilde p_\varepsilon\,\tilde \varPhi\, dx'dy_3\right|&\leq\displaystyle \mu \left\Vert D_{\varepsilon}\tilde{u}_{\varepsilon}\right\Vert_{L^2(\widetilde{\Omega}_{\varepsilon})^{{3\times3}}}\left\Vert D_{\varepsilon}\tilde{\varphi}\right\Vert_{L^2(\widetilde{\Omega}_{\varepsilon})^{{3\times3}}}+\alpha \varepsilon^{\gamma}\left|\int_{\partial T_\varepsilon}\tilde u_\varepsilon \cdot \tilde \varphi\, d\sigma(x')dy_3\right|\\
\noame
&\displaystyle + C\left\Vert \tilde{\varphi}\right\Vert_{L^2(\widetilde{\Omega}_{\varepsilon})^{3}}+\left|\int_{\partial T_\varepsilon} g'_\varepsilon \cdot \tilde\varphi'\, d\sigma(x')dy_3\right|\,.
\end{array}
\end{equation}
By using that $ g \in L^2(\partial T)^3$ is a $Y'$-periodic function and estimate (\ref{gvestimate}), we have
$$
\left|\int_{\partial T_\varepsilon} g'_\varepsilon\,\cdot \tilde \varphi'\, d\sigma(x')dy_3\right|\leq C\left(\varepsilon^{-1}\|\tilde \varphi\|_{L^2(\widetilde\Omega_\varepsilon)^3} + \|D_\varepsilon \tilde \varphi\|_{L^2(\widetilde \Omega_\varepsilon)^{3\times 3}}\right)\,.
$$
Analogously, using estimate (\ref{gvestimate_L1}) and the Cauchy- Schwarz inequality, a simple computation gives
$$
\begin{array}{rl}
\displaystyle
\alpha \varepsilon^{\gamma}\left|\int_{\partial  T_\varepsilon}\tilde u_\varepsilon \cdot \tilde \varphi\, d\sigma(x')dy_3\right|\leq &\displaystyle  \varepsilon^{\gamma-1}\, C\, \| \tilde u_\varepsilon \|_{L^2(\widetilde \Omega_\varepsilon)}\| \tilde \varphi \|_{L^2(\widetilde \Omega_\varepsilon)}+\varepsilon^{\gamma}\, C\, \| \tilde u_\varepsilon \|_{L^2(\widetilde \Omega_\varepsilon)}\| D_\varepsilon\tilde \varphi \|_{L^2(\widetilde \Omega_\varepsilon)}
\\
\noame &
\displaystyle
+\,\varepsilon^{\gamma}\, C\, \| D_\varepsilon\tilde u_\varepsilon \|_{L^2(\widetilde \Omega_\varepsilon)}\| \tilde \varphi \|_{L^2(\widetilde \Omega_\varepsilon)}.
\end{array}
$$
Then, turning back to (\ref{5.11_new}) and using (\ref{5.10_new}), one has 
\begin{equation}\label{final_presion}
\begin{array}{rl}
\displaystyle
\left|\int_{\widetilde \Omega_\varepsilon}\tilde p_\varepsilon\,\tilde \varPhi\, dx'dy_3\right|&\leq\displaystyle C\left(\varepsilon^{-1} +\varepsilon^{\gamma}\right) \left\Vert D_{\varepsilon}\tilde{u}_{\varepsilon}\right\Vert_{L^2(\widetilde{\Omega}_{\varepsilon})^{{3\times3}}}\|\tilde \varPhi \|_{L^2(\widetilde{\Omega}_{\varepsilon})}
\\
\noame
&\displaystyle +C\left(\varepsilon^{\gamma-1} \left\Vert \tilde{u}_{\varepsilon}\right\Vert_{L^2(\widetilde{\Omega}_{\varepsilon})^{3}}+\varepsilon^{-1}\right)\|\tilde \varPhi \|_{L^2(\widetilde{\Omega}_{\varepsilon})}\,.
\end{array}
\end{equation}
The {\it a priori} estimates for the pressure follow now from (\ref{final_presion}) and estimates (\ref{a3})-(\ref{a3-2}) and (\ref{a3-3}), corresponding to the different values of $\gamma$.

\end{proof}

\paragraph*{A priori estimates of the unfolding functions $(\hat u_\varepsilon,\hat p_\varepsilon)$:} let us obtain some {\it a priori} estimates for the sequences $(\hat{u}_{\varepsilon}, \hat{p}_{\varepsilon})$ where $\hat u_\varepsilon$ and $\hat p_\varepsilon$ are obtained by applying the change of variable (\ref{vhat}) to $(\tilde u_\varepsilon,\tilde p_\varepsilon)$.
\begin{lemma}\label{estCV}
We distinguish three cases:
\begin{itemize}
\item[i) ] for $\gamma<-1$, then there exists a constant $C$ independent of $\varepsilon$, such that 
\begin{equation}\label{estim_velocidad_gorro1_gamma_case_1}
\left\Vert \hat{u}_{\varepsilon}\right\Vert_{L^2(\mathbb{R}^2\times Y_f)^3}\leq C {\varepsilon},\quad \left\Vert D_{y}\hat{u}_{\varepsilon}\right\Vert_{L^2(\mathbb{R}^2\times Y_f)^{{3\times 3}}}\leq C{\varepsilon},
\end{equation}
\begin{equation}\label{esti_gorro}
\left\Vert \hat{p}_{\varepsilon} \right\Vert_{L^{2}(\mathbb{R}^2\times Y_f)}\leq C\varepsilon^{\gamma}.
\end{equation}
\item[ii) ] for $-1\leq \gamma<1$, then there exists a constant $C$ independent of $\varepsilon$, such that
\begin{equation}\label{estim_velocidad_gorro1_gamma_case_2}
\left\Vert \hat{u}_{\varepsilon}\right\Vert_{L^2(\mathbb{R}^2\times Y_f)^3}\leq C {\varepsilon}^{-\gamma},\quad \left\Vert D_{y}\hat{u}_{\varepsilon}\right\Vert_{L^2(\mathbb{R}^2\times Y_f)^{{3\times 3}}}\leq C{\varepsilon}^{1-\gamma \over 2},
\end{equation}
\begin{equation}\label{esti_gorro2}
\left\Vert \hat{p}_{\varepsilon} \right\Vert_{L^{2}(\mathbb{R}^2\times Y_f)}\leq C\varepsilon^{-1}.
\end{equation}
\item[iii) ] for $\gamma\geq 1$, then there exists a constant $C$ independent of $\varepsilon$, such that
\begin{equation}\label{estim_velocidad_gorro1_gamma_case_3}
\left\Vert \hat{u}_{\varepsilon}\right\Vert_{L^2(\mathbb{R}^2\times Y_f)^3}\leq C {\varepsilon}^{-1},\quad \left\Vert D_{y}\hat{u}_{\varepsilon}\right\Vert_{L^2(\mathbb{R}^2\times Y_f)^{{3\times 3}}}\leq C,
\end{equation}
\begin{equation}\label{esti_gorro3}
\left\Vert \hat{p}_{\varepsilon} \right\Vert_{L^{2}(\mathbb{R}^2\times Y_f)}\leq C\varepsilon^{-2}.
\end{equation}
\end{itemize}
\end{lemma}
\begin{proof} Using properties (\ref{normv}) and (\ref{normDv}) with $p=2$ and the {\it a priori} estimates given in Lemma \ref{Lemma_a3} and Lemma \ref{Lemma_a3_P}, we have the desired result.

\end{proof}

\subsection{Some compactness results}\label{s5}
Let us remember that, for the velocity,  we denote by $\tilde U_\varepsilon$ the zero extension of $\tilde u_\varepsilon$ to the whole $\Omega$, and for the pressure we denote  by $\tilde P_\varepsilon$ the zero extension of $\tilde p_\varepsilon$ to the whole $\Omega$. In this subsection we obtain some compactness results about the behavior of the sequences $\tilde \varPi_\varepsilon \tilde u_\varepsilon$, where $\tilde \Pi_\varepsilon$ is given in Corollary \ref{corollary_para_Poincare}, $(\tilde U_\varepsilon, \tilde P_\varepsilon)$ and $(\hat{u}_{\varepsilon}, \hat{p}_{\varepsilon})$.

\begin{lemma}
There exists an extension operator  $\tilde \Pi_\varepsilon$, given in Corollary \ref{corollary_para_Poincare}, such that
\begin{itemize}
\item[i) ] for $\gamma<-1$, then 
 \begin{equation}\label{convtilde1}
\tilde \varPi_\varepsilon \tilde u_\varepsilon\rightharpoonup 0 \text{\ in \ }H^1_0(\Omega)^3.
\end{equation}
Moreover, $(\varepsilon^{-1}\tilde U_\varepsilon, \varepsilon^{-\gamma}\tilde P_\varepsilon)$ is bounded in $H^1(0,1;L^2(\omega)^3)\times L^2(\Omega)$ and any weak-limit point $(\tilde u, \tilde p)$ of this sequence satisfies
\begin{equation}\label{convtilde1-omega}
\varepsilon^{-1}\tilde{U}_{\varepsilon}\rightharpoonup \tilde{u}\text{\ in \ }H^1(0,1;L^2(\omega)^3),
\end{equation}
\begin{equation}\label{convPtilde1}
\varepsilon^{-\gamma}\tilde{P}_{\varepsilon}\rightharpoonup \tilde{p}\text{\ in \ }L^2(\Omega),
\end{equation}
\item[ii) ] for $-1\leq \gamma<1$, then 
\begin{equation}\label{convtilde_nuevo_1}
\varepsilon^{{\gamma+1\over 2}}\tilde \varPi_\varepsilon \tilde u_\varepsilon\rightharpoonup 0\text{\ in \ }H^1_0(\Omega)^3.
\end{equation}
Moreover, $(\varepsilon^{\gamma}\tilde U_\varepsilon, \varepsilon \tilde P_\varepsilon)$ is bounded in $H^1(0,1;L^2(\omega)^3)\times L^2(\Omega)$ and any weak-limit point $(\tilde u, \tilde p)$ of this sequence satisfies
\begin{equation}\label{convtilde_nuevo_1-omega}
\varepsilon^{{\gamma}}\tilde{U}_{\varepsilon}\rightharpoonup \tilde{u}\text{\ in \ }H^1(0,1;L^2(\omega)^3),
\end{equation}
\begin{equation}\label{convPtilde1-case2}
\varepsilon\tilde{P}_{\varepsilon}\rightharpoonup \tilde{p}\text{\ in \ }L^2(\Omega),
\end{equation}
\item[iii) ] for $\gamma\geq 1$, then
\begin{equation}\label{convtilde_nuevo_3}
\varepsilon \tilde \varPi_\varepsilon \tilde u_\varepsilon\rightharpoonup 0\text{\ in \ }H^1_0(\Omega)^3.
\end{equation}
Moreover, $(\tilde U_\varepsilon, \varepsilon^2 \tilde P_\varepsilon)$ is bounded in $H^1(0,1;L^2(\omega)^3)\times L^2(\Omega)$ and any weak-limit point $(\tilde u, \tilde p)$ of this sequence satisfies
\begin{equation}\label{convtilde_nuevo_3-omega}
\tilde{U}_{\varepsilon}\rightharpoonup \tilde{u}\text{\ in \ }H^1(0,1;L^2(\omega)^3),
\end{equation}
\begin{equation}\label{convPtilde1-case3}
\varepsilon^2\tilde{P}_{\varepsilon}\rightharpoonup \tilde{p}\text{\ in \ }L^2(\Omega).
\end{equation}
\end{itemize}

\end{lemma}

\begin{proof} 
We proceed in three steps: 

{\it Step 1.} For $\gamma< -1$: from estimates (\ref{a3}) and (\ref{Estim_pep_1}), we have immediately the convergences, after eventual extraction of subsequences,  (\ref{convtilde1-omega}) and (\ref{convPtilde1}).

Moreover, we have
$$||\tilde u_\varepsilon||_{L^2(\widetilde\Omega_\varepsilon)^3}\leq C,$$
and we can apply Corollary \ref{corollary_para_Poincare} to $\tilde u_\varepsilon$ and  we deduce the existence of $u^*\in H^1_0(\Omega)^3$ such that $\tilde \varPi_\varepsilon \tilde u_\varepsilon$ converges weakly to $u^*$ in $H^1_0(\Omega)^3$. Consequently, by Rellich theorem, $\tilde \varPi_\varepsilon \tilde u_\varepsilon$ converges strongly to $u^*$ in $L^2(\Omega)^3$.

On the other side, we have the following indentity:
$$\begin{array}{l}
\chi_{\widetilde\Omega_\varepsilon} \left(\tilde \varPi_\varepsilon \tilde u_\varepsilon\right)=\varepsilon\, \varepsilon^{-1}\tilde U_\varepsilon \quad \text{\ in \ } \Omega.
\end{array}$$
Due the periodicity of the domain $\widetilde \Omega_\varepsilon$ we have that $\chi_{\widetilde\Omega_\varepsilon}$ converges weakly-$\star$ to  ${|Y'_f|\over |Y'|}$ in $L^\infty(\Omega)$, and we can pass to the limit in the term of the left hand side. Thus, $\chi_{\widetilde\Omega_\varepsilon} \left(\tilde \varPi_\varepsilon \tilde u_\varepsilon\right)$ converges weakly to ${|Y'_f|\over |Y'|} u^*$ in $L^2(\Omega)^3$. In the right hand side, $\tilde{U}_{\varepsilon}$ converges to zero, so we obtain (\ref{convtilde1}).

{\it Step 2.} For $-1\leq \gamma<1$: from the  estimates (\ref{a3-2}) and (\ref{Estim_pep_2}), we deduce convergences (\ref{convtilde_nuevo_1-omega}) and (\ref{convPtilde1-case2}).

Moreover, as $-1\leq \gamma<1$, we have
$$\|\varepsilon^{\gamma+1\over 2}\tilde u_\varepsilon\|_{L^2(\widetilde\Omega_\varepsilon)^3}\leq C,$$
and using Corollary  \ref{corollary_para_Poincare}, we have 
$$
\varepsilon^{{\gamma+1\over 2}}\tilde \varPi_\varepsilon \tilde u_\varepsilon\rightharpoonup u^*\text{\ in \ }H^1(\Omega)^3.
$$
Consequently, 
$$
\varepsilon^{{\gamma+1\over 2}}\tilde \varPi_\varepsilon \tilde u_\varepsilon\to u^*\text{\ in \ }L^2(\Omega)^3,
$$
and passing to the limit in the identity 
$$\begin{array}{l}
\chi_{\widetilde\Omega_\varepsilon} \left(\varepsilon^{\gamma+1\over 2}\tilde \varPi_\varepsilon\tilde u_\varepsilon\right)=\varepsilon^{1-\gamma\over 2} \varepsilon^{\gamma}\tilde U_\varepsilon \quad \text{\ in \ } \Omega,
\end{array}$$
 we deduce that $u^*=0$, and so (\ref{convtilde_nuevo_1}) is proved.

{\it Step 3.} For $\gamma\geq 1$: from estimate (\ref{a3-3}) and Dirichlet boundary condition, we deduce that $$||\tilde U_\varepsilon||_{L^2( \Omega)^3}\leq ||\partial_{y_3}\tilde U_\varepsilon||_{L^2(\Omega)^3}\leq C,$$
and we have immediately, after eventual extraction of subsequences, the convergence (\ref{convtilde_nuevo_3-omega}). From estimate (\ref{Estim_pep_3}), we have immediately, after eventual extraction of subsequences, the convergence (\ref{convPtilde1-case3}).

Moreover, we can apply Corollary \ref{corollary_para_Poincare} to $\tilde u_\varepsilon$ and  we deduce the existence of $u^*\in H^1_0(\Omega)^3$ such that $\varepsilon\tilde \varPi_\varepsilon \tilde u_\varepsilon$ converges weakly to $u^*$ in $H^1_0(\Omega)^3$. Consequently, by Rellich theorem, $\varepsilon\tilde \varPi_\varepsilon \tilde u_\varepsilon$ converges strongly to $u^*$ in $L^2(\Omega)^3$.

On the other side, we have the following indentity:
$$\begin{array}{l}
\chi_{\widetilde\Omega_\varepsilon} \left(\varepsilon\tilde \varPi_\varepsilon \tilde u_\varepsilon\right)=\varepsilon\,\tilde U_\varepsilon \quad \text{\ in \ } \Omega.
\end{array}$$
We can pass to the limit in the term of the left hand side. Thus, $\chi_{\widetilde\Omega_\varepsilon} \left(\varepsilon\tilde \varPi_\varepsilon \tilde u_\varepsilon\right)$ converges weakly to ${|Y'_f|\over |Y'|} u^*$ in $L^2(\Omega)^3$. In the right hand side, $\varepsilon\,\tilde{U}_{\varepsilon}$ converges to zero, so we obtain (\ref{convtilde_nuevo_3}).

\end{proof}
%

Finally, we give a convergence result for $\hat{u}_{\varepsilon}$.
\begin{lemma}\label{lemma_gorro}
We distinguish three cases:
\begin{itemize}
\item[i) ] for $\gamma<-1$, then for a subsequence of $\varepsilon$ still denote by $\varepsilon$, there exists $\hat{u}\in L^2(\mathbb{R}^2;H^1_{\sharp}(Y_f)^3)$, such that
\begin{equation}\label{convUgorro1}
 \varepsilon^{-1} \hat u_\varepsilon\rightharpoonup \hat u\quad\hbox{in }L^2(\mathbb{R}^2;H^1(Y_f)^3),
\end{equation}
\begin{equation}\label{convUgorrofrontera}
 \varepsilon^{-1} \hat u_\varepsilon\rightharpoonup \hat u\quad\hbox{in }L^2(\mathbb{R}^2;H^{1\over 2}(\partial T)^3),
\end{equation}
\begin{equation}\label{propertygrau0}
{|Y'_f| \over | Y'|}\mathcal{M}_{Y'_f} [\hat u]= \tilde u \quad \hbox{a.e. in }\Omega,
\end{equation}
\item[ii) ] for $-1\leq \gamma<1$, then for a subsequence of $\varepsilon$ still denote by $\varepsilon$, there exists there exists $\hat{u}\in L^2(\mathbb{R}^2;H^1_{\sharp}(Y_f)^3)$, independent of $y$, such that
\begin{equation}\label{convUgorro2}
\varepsilon^{\gamma} \hat u_\varepsilon\rightharpoonup \hat u\quad\hbox{in }L^2(\mathbb{R}^2;H^1(Y_f)^3),
\end{equation}
\begin{equation}\label{convUgorrofrontera2}
 \varepsilon^{\gamma} \hat u_\varepsilon\rightharpoonup \hat u\quad\hbox{in }L^2(\mathbb{R}^2;H^{1\over 2}(\partial T)^3),
\end{equation}
\begin{equation}\label{propertygrau2}
{|Y'_f|\over |Y'|}\hat u=\tilde u\quad \hbox{a.e. in }\Omega,
\end{equation}
\item[iii) ] for $\gamma\geq 1$, then for a subsequence of $\varepsilon$ still denote by $\varepsilon$, there exists there exists $\hat{u}\in L^2(\mathbb{R}^2;H^1_{\sharp}(Y_f)^3)$ such that

\begin{equation}\label{convUgorrofrontera3}
\hat u_\varepsilon- \mathcal{M}_{Y_f}[\hat u_\varepsilon]\rightharpoonup \hat u\quad\hbox{in }L^2(\mathbb{R}^2;H^{1\over 2}(\partial T)^3).
\end{equation}
\begin{equation}\label{convUgorro3Derivada}
D_y \hat u_\varepsilon\rightharpoonup D_y\hat u\quad\hbox{in }L^2(\mathbb{R}^2\times Y_f)^{3\times 3}.
\end{equation}
\begin{equation}\label{propertygrau3}
{|Y'_f| \over | Y'|}\mathcal{M}_{Y'_f} [\hat u]= \tilde u \quad \hbox{a.e. in }\Omega,
\end{equation}
\begin{equation}\label{incomp_div_hatu}
{\rm div}_y\hat u=0\quad \hbox{ in } \omega\times Y_f.
\end{equation}
\end{itemize}
\end{lemma}
\begin{proof}
We proceed in three steps: 

{\it Step 1.} For $\gamma< -1$, using (\ref{estim_velocidad_gorro1_gamma_case_1}), there exists $\hat u: \mathbb{R}^2\times Y_f\to \mathbb{R}^3$, such that convergence (\ref{convUgorro1}) holds. Passing to the limit by semicontinuity and using estimates (\ref{estim_velocidad_gorro1_gamma_case_1}), we get
$$\int_{\mathbb{R}^2 \times Y_f}\left\vert \hat{u} \right\vert^2dx^{\prime}dy\leq C,\quad \int_{\mathbb{R}^2 \times Y_f}\left\vert D_{y}\hat{u} \right\vert^2dx^{\prime}dy\leq C,$$ which, once we prove the $Y'$-periodicity of $\hat u$ in $y'$, shows that $\hat{u}\in L^2(\mathbb{R}^2;H^1_\sharp(Y_f)^3)$.

It remains to prove the $Y^{\prime}$-periodicity of $\hat{u}$ in $y^{\prime}$. To do this, we observe that by definition of $\hat{u}_{\varepsilon}$ given by (\ref{vhat}) applied to $\tilde u_\varepsilon$, we have
$$
\hat{u}_{\varepsilon}(x_1+\varepsilon,x_2,-1/2,y_2,y_3)=\hat{u}_{\varepsilon}(x',1/2,y_2,y_3) \hbox{ a.e. }(x',y_2,y_3)\in
\mathbb{R}^2\times (-1/2,1/2)\times (0,1)\,.
$$
Multiplying by $\varepsilon^{-1}$ and passing to the limit by (\ref{convUgorro1}), we get
$$
\hat{u}(x',-1/2,y_2,y_3)=\hat{u}(x',1/2,y_2,y_3) \hbox{ a.e. }(x',y_2,y_3)\in
\mathbb{R}^2\times (-1/2,1/2)\times (0,1)\,.
$$
Analogously, we can prove
$$
\hat{u}(x',y_1,-1/2,y_3)=\hat{u}(x',y_1,1/2,y_3) \hbox{ a.e. }(x',y_1,y_3)\in
\omega\times (-1/2,1/2)\times (0,1)\,.
$$
These equalities prove that $\hat{u}$ is periodic with respect to $Y'$. Convergence (\ref{convUgorrofrontera}) is straightforward from the definition (\ref{vhat_b}) and the Sobolev injections.

Finally, using Proposition \ref{properties_unf}, we can deduce
\begin{equation*}
{1 \over | Y'|}\int_{\mathbb{R}^2 \times Y_f} \hat u_\varepsilon(x',y)dx'dy=\int_{\widetilde \Omega_\varepsilon} \tilde u_\varepsilon(x',y_3) dx'dy_3,
\end{equation*}
and multiplying by $\varepsilon^{-1}$ and taking into account that $\tilde u_\varepsilon$ is extended by zero to the whole $\Omega$, we have
\begin{equation*}
{1 \over \varepsilon| Y'|}\int_{\mathbb{R}^2 \times Y_f} \hat u_\varepsilon(x',y)dx'dy={1 \over \varepsilon}\int_{ \Omega} \tilde U_\varepsilon(x',y_3) dx'dy_3.
\end{equation*}
Using convergences (\ref{convtilde1-omega}) and (\ref{convUgorro1}), we have (\ref{propertygrau0}).

{\it Step 2.} For $-1\leq \gamma<1$, using (\ref{estim_velocidad_gorro1_gamma_case_2}) and taking into account that $\varepsilon^{1-\gamma \over 2}\leq \varepsilon^{-\gamma}$, then there exists $\hat u: \mathbb{R}^2\times Y_f\to \mathbb{R}^3$, such that convergence (\ref{convUgorro2}) holds. Convergence (\ref{convUgorrofrontera2}) is straightforward from the definition (\ref{vhat_b}) and the Sobolev injections.

On the other hand, since $\varepsilon^{\gamma} D_y \hat u_\varepsilon=\varepsilon^{\gamma+1 \over 2} \varepsilon^{\gamma-1 \over 2}D_y \hat u_\varepsilon$ and $\varepsilon^{\gamma-1 \over 2}D_y \hat u_\varepsilon$ is bounded in $L^2(\mathbb{R}^2 \times Y_f)^{3\times 3}$, using (\ref{convUgorro2}) we can deduce that $D_y \hat u=0$. This implies that $\hat u$ is independent of $y$. Finally, (\ref{propertygrau2}) is obtained similarly to the Step 1.

{\it Step 3.} For $\gamma\geq 1$, using the second {\it a priori} estimate in (\ref{estim_velocidad_gorro1_gamma_case_3}) and the Poincar\'e-Wirtinger inequality $$\int_{Y_f}\left|\hat u_\varepsilon-\mathcal{M}_{Y_f}[\hat u_\varepsilon] \right|^2 dy\leq C\int_{Y_f}| D_y \hat u_\varepsilon|^2 dy, \quad \text{a.e. in }\omega,$$ 
we deduce that there exists $\hat u \in L^2(\mathbb{R}^2;H^1(Y_f)^3)$ such that 
$$\hat U_\varepsilon=\hat u_\varepsilon-\mathcal{M}_{Y_f}[\hat u_\varepsilon] \rightharpoonup \hat u\quad \text{ in } L^2(\mathbb{R}^2;H^1(Y_f)^3),$$
and (\ref{convUgorro3Derivada}) holds. Convergence (\ref{convUgorrofrontera3}) is straightforward from the definition (\ref{vhat_b}) and the Sobolev injections.

It remains to prove the $Y^{\prime}$-periodicity of $\hat{u}$ in $y^{\prime}$. To do this, we observe that by definition of $\hat{u}_{\varepsilon}$ given by (\ref{vhat}) applied to $\tilde u_\varepsilon$, we have
$$
\hat{u}_{\varepsilon}(x_1+\varepsilon,x_2,-1/2,y_2,y_3)=\hat{u}_{\varepsilon}(x',1/2,y_2,y_3) \hbox{ a.e. }(x',y_2,y_3)\in
\mathbb{R}^2\times (-1/2,1/2)\times (0,1)\,,
$$
which implies 
$$
\hat{U}_\varepsilon(x',-1/2,y_2,y_3)-\hat{U}_\varepsilon(x',1/2,y_2,y_3)=-\mathcal{M}_{Y_f}[\hat u_\varepsilon](x'+\varepsilon\,e_1)+\mathcal{M}_{Y}[\hat u_\varepsilon](x'), $$
which tends to zero (see the proof of Proposition 2.8 in \cite{CioDonZak3}), and so $$
\hat{u}(x',-1/2,y_2,y_3)=\hat{u}(x',1/2,y_2,y_3) \hbox{ a.e. }(x',y_2,y_3)\in
\mathbb{R}^2\times (-1/2,1/2)\times (0,1)\,.
$$
Analogously, we can prove
$$
\hat{u}(x',y_1,-1/2,y_3)=\hat{u}(x',y_1,1/2,y_3) \hbox{ a.e. }(x',y_1,y_3)\in
\omega\times (-1/2,1/2)\times (0,1)\,.
$$
These equalities prove that $\hat{u}$ is periodic with respect to $Y'$.

{\it Step 4.} From the second variational formulation in (\ref{fv2}), by Proposition \ref{properties_unf_b}, we have that 
\begin{equation}\label{varform_incompress_lemma}\int_{\widetilde \Omega_\varepsilon}\left(\tilde u'_\varepsilon\cdot \nabla_{x'}\tilde \psi+ \varepsilon^{-1}\tilde u_{\varepsilon,3}\partial_{y_3}\tilde \psi\right)\,dx'dy_3= {\varepsilon^{-1} \over |Y'|}\int_{\omega \times \partial  T}(\hat u_\varepsilon \cdot  n)\hat \psi_\varepsilon\, dx'd\sigma(y')dy_3,\quad \forall \,\tilde \psi\in \widetilde H_\varepsilon,
\end{equation}
and using the extension of the velocity, we obtain 
$$\int_{\Omega}\left(\tilde U'_\varepsilon\cdot \nabla_{x'}\tilde \psi+ \varepsilon^{-1}\tilde U_{\varepsilon,3}\partial_{y_3}\tilde \psi\right)\,dx'dy_3= {\varepsilon^{-1} \over |Y'|}\int_{\omega \times \partial  T}(\hat u_\varepsilon \cdot n) \hat \psi_\varepsilon\, dx'd\sigma(y')dy_3,\quad \forall \,\tilde \psi\in \widetilde H_\varepsilon.
$$ 
We remark that the second term in the left-hand side and the one in the right-hand side  have the same order, so in every cases when passing to the limit after multiplying by a suitable power of $\varepsilon$ and using that $\hat \psi_\varepsilon$ converges strongly to $\tilde \psi$ in $L^2(\omega \times \partial T)$ (see Proposition 2.6 in \cite{CioDonZak3} for more details), it is not possible to find the usual  incompressibility condition in thin  domains given by
$${\rm div}_{x'}\left(\int_0^1 \tilde u'(x',y_3)\,dy_3\right)=0\quad \hbox{on }\omega.$$

On the other hand, we focus in the third case. Thus, using Proposition \ref{properties_unf} in the left-hand side of (\ref{varform_incompress_lemma}), we have
\begin{equation}\label{varform_incompress_lemma_2}{\varepsilon^{-1} \over |Y'|}\int_{\omega\times Y_f}\hat u_\varepsilon\cdot \nabla_{y}\hat \psi_\varepsilon\,dx'dy= {\varepsilon^{-1} \over |Y'|}\int_{\omega \times \partial  T}\hat u_\varepsilon \cdot \hat \psi_\varepsilon\, dx'd\sigma(y')dy_3,
\end{equation}
which, multiplying by $\varepsilon|Y'|$ and since $\mathcal{M}_{Y_f}[\hat u_\varepsilon]$ does not depend on the variable $y$,  is equivalent to 
\begin{equation}\label{varform_incompress_lemma_3}\int_{\omega\times Y_f}(\hat u_\varepsilon-\mathcal{M}_{Y_f}[\hat u_\varepsilon])\cdot \nabla_y \hat \psi_\varepsilon\,dx'dy=\int_{\omega \times \partial  T}[(\hat u_\varepsilon-\mathcal{M}_{Y_f}[\hat u_\varepsilon])\cdot n] \cdot \hat \psi_\varepsilon\, dx'd\sigma(y')dy_3
\end{equation}
Thus, passing to the limit using convergences  (\ref{convUgorro3Derivada}), we get condition (\ref{incomp_div_hatu}).
\end{proof}

\subsection{Proof of Theorem \ref{MainTheorem}}\label{s6}
In this section, we will multiply system (\ref{fv2}) by a test function having the form of the limit $\hat{u}$ (as explained in Lemma \ref{lemma_gorro}), and we will use the convergences given in the previous section in order to identify the homogenized model in every cases. 
\paragraph*{Proof of Theorem \ref{MainTheorem}:}
We proceed in three steps:

{\it Step 1.} For $\gamma< -1$. Let $\tilde \varphi \in \mathcal{D}(\Omega)^3$ and $\tilde \psi \in \mathcal{D}(\Omega)$ be  test functions in (\ref{fv2}). By Proposition \ref{properties_unf_b}, one has
\begin{equation*}\label{case1limit}
\begin{array}{l}
\displaystyle
\mu \int_{\widetilde \Omega_\varepsilon}D_\varepsilon \tilde u_\varepsilon:D_\varepsilon \tilde \varphi\, dx'dy_3-\int_{\widetilde \Omega_\varepsilon}\tilde p_\varepsilon\,{\rm div}_\varepsilon\,\tilde \varphi\, dx'dy_3+\alpha {\varepsilon^{\gamma-1} \over |Y'|}\int_{\omega \times \partial  T}\hat u_\varepsilon \cdot \hat \varphi_\varepsilon\, dx'd\sigma(y) \\
\noame \displaystyle=\int_{\widetilde \Omega_\varepsilon}f'_\varepsilon\cdot \tilde \varphi'\,dx'dy_3+{\varepsilon^{-1} \over |Y'|}\int_{\omega \times \partial  T} \tilde g' \cdot \hat\varphi'_\varepsilon\, dx'd\sigma(y),
\end{array}
\end{equation*}
i.e.,
\begin{equation}\label{fv_unfolding}
\begin{array}{l}
\displaystyle
\mu \int_{\widetilde \Omega_\varepsilon}D_{x'} \tilde u_\varepsilon:D_{x'} \tilde \varphi\, dx'dy_3+{\mu \over \varepsilon^2}\int_{\widetilde \Omega_\varepsilon}\partial_{y_3} \tilde u_\varepsilon\cdot \partial_{y_3} \tilde \varphi\, dx'dy_3\\
\noame\displaystyle-\int_{\widetilde \Omega_\varepsilon}\tilde p_\varepsilon\,{\rm div}_{x'}\,\tilde \varphi'\, dx'dy_3-{1 \over \varepsilon}\int_{\widetilde \Omega_\varepsilon}\tilde p_\varepsilon\,\partial_{y_3}\,\tilde \varphi_3\, dx'dy_3+\alpha {\varepsilon^{\gamma-1} \over |Y'|}\int_{\omega \times \partial  T}\hat u_\varepsilon \cdot \hat \varphi_\varepsilon\, dx'd\sigma(y)\\
\noame 
\displaystyle
 =\int_{\widetilde \Omega_\varepsilon}f'_\varepsilon\cdot \tilde \varphi'\,dx'dy_3+{\varepsilon^{-1} \over |Y'|}\int_{\omega \times \partial  T} \tilde g' \cdot \hat\varphi'_\varepsilon\, dx'd\sigma(y),
\end{array}
\end{equation}

Next, we prove that $\tilde p$ does not depend on $y_3$. Let $\tilde \varphi=(0,\varepsilon^{-\gamma+1}\tilde \varphi_3) \in \mathcal{D}(\Omega)^3$ be a test function in (\ref{fv_unfolding}), we have
\begin{equation*}
\begin{array}{l}
\displaystyle
\mu\, \varepsilon^{-\gamma+1}\int_{\widetilde \Omega_\varepsilon}\nabla_{x'} \tilde u_{\varepsilon,3}\cdot \nabla_{x'}\tilde \varphi_3\, dx'dy_3+\mu \,\varepsilon^{-\gamma-1}\int_{\widetilde \Omega_\varepsilon}\partial_{y_3} \tilde u_{\varepsilon,3}\partial_{y_3} \tilde \varphi_3\, dx'dy_3\\
\noame
\displaystyle-\varepsilon^{-\gamma}\int_{\widetilde \Omega_\varepsilon}\tilde p_\varepsilon\,\partial_{y_3}\tilde \varphi_3\, dx'dy_3+ {\alpha \over |Y'|}\int_{\omega \times \partial  T}\hat u_{\varepsilon,3} \cdot \hat \varphi_{\varepsilon,3}\, dx'd\sigma(y) =0.
\end{array}
\end{equation*}
Taking into account that $\tilde P_\varepsilon$ is zero extension of $\tilde p_\varepsilon$ to the whole $\Omega$, we have
\begin{equation*}
\int_{\widetilde \Omega_\varepsilon}\tilde p_\varepsilon\,\partial_{y_3}\tilde \varphi_3\, dx'dy_3=\int_{ \Omega}\tilde P_\varepsilon\,\partial_{y_3}\tilde \varphi_3\, dx'dy_3,
\end{equation*}
and by the second {\it a priori} estimate (\ref{a3}), the convergences (\ref{convPtilde1}) and (\ref{convUgorrofrontera}), passing to the limit we have
\begin{equation*}
\int_{ \Omega}\tilde p\,\partial_{y_3}\tilde \varphi_3\, dx'dy_3=0,
\end{equation*}
so $\tilde p$ does not depend on $y_3$.

Let $\tilde \varphi=(\varepsilon^{-\gamma}\varphi'(x',y_3),\varepsilon^{-\gamma}\tilde \varphi_3(x')) \in \mathcal{D}(\Omega)^3$ be a test function in (\ref{fv_unfolding}), we have
\begin{equation*}
\begin{array}{l}
\displaystyle
\mu \varepsilon^{-\gamma}\int_{\widetilde \Omega_\varepsilon}D_{x'} \tilde u'_\varepsilon:D_{x'} \tilde \varphi'\, dx'dy_3+\mu \varepsilon^{-\gamma-2}\int_{\widetilde \Omega_\varepsilon}\partial_{y_3} \tilde u'_\varepsilon\cdot \partial_{y_3} \tilde \varphi'\, dx'dy_3-\varepsilon^{-\gamma}\int_{\widetilde \Omega_\varepsilon}\tilde p_\varepsilon\,{\rm div}_{x'}\,\tilde \varphi'\, dx'dy_3\\
\noame 
\displaystyle
+\alpha {\varepsilon^{-1} \over |Y'|}\int_{\omega \times \partial  T}\hat u'_\varepsilon \cdot \hat \varphi'_\varepsilon\, dx'd\sigma(y) =\varepsilon^{-\gamma}\int_{\widetilde \Omega_\varepsilon}f'_\varepsilon\cdot \tilde \varphi'\,dx'dy_3+{\varepsilon^{-\gamma-1} \over |Y'|}\int_{\omega \times \partial  T} \tilde g' \cdot \hat\varphi'_\varepsilon\, dx'd\sigma(y),
\end{array}
\end{equation*}
and
\begin{equation*}
\begin{array}{l}
\displaystyle
\mu\, \varepsilon^{-\gamma} \int_{\widetilde \Omega_\varepsilon}\nabla_{x'} \tilde u_{\varepsilon,3}\cdot \nabla_{x'} \tilde \varphi_3\, dx'dy_3
+\alpha {\varepsilon^{-1} \over |Y'|}\int_{\omega \times \partial  T}\hat u_{\varepsilon,3} \, \hat \varphi_{\varepsilon,3}\, dx'd\sigma(y) =0.
\end{array}
\end{equation*}
Taking into account that $\tilde P_\varepsilon$ is zero extension of $\tilde p_\varepsilon$ to the whole $\Omega$, we have
\begin{equation*}
\int_{\widetilde \Omega_\varepsilon}\tilde p_\varepsilon\,{\rm div}_{x'}\,\tilde \varphi'\, dx'dy_3=\int_{ \Omega}\tilde P_\varepsilon\,{\rm div}_{x'}\,\tilde \varphi'\, dx'dy_3.
\end{equation*}
Using that $\hat \varphi_\varepsilon$ converges strongly to $\tilde \varphi$ in $L^2(\omega \times \partial T)^3$ (see Proposition 2.6 in \cite{CioDonZak3} for more details) and by the second {\it a priori} estimate (\ref{a3}), the convergences (\ref{convPtilde1}) and (\ref{convUgorrofrontera}), passing to the limit we have
\begin{equation*}
-\int_{\Omega}\tilde p(x')\, {\rm div}_{x'}\,\tilde \varphi'(x',y_3)dx'dy_3+{\alpha \over |Y'|}\int_{\omega \times \partial T'}\int_0^1\hat u'(x',y) \cdot \tilde \varphi'(x',y_3) dx'd\sigma(y')dy_3=0,
\end{equation*}
and
\begin{equation*}
{\alpha \over |Y'|}\int_{\omega \times \partial T}\hat u_3(x',y) \, \tilde \varphi_3(x') \,dx'd\sigma(y)=0,
\end{equation*}
which implies that $\mathcal{M}_{\partial T'}[\hat u_3]=0$.

Taking into account that 
$$\int_{\omega \times \partial T'}\int_0^1\hat u'(x',y) \cdot \tilde \varphi'(x',y_3) dx'd\sigma(y')dy_3=|\partial T'|\int_\Omega \mathcal{M}_{\partial T'}[\hat u'](x',y_3)\cdot \tilde \varphi'(x',y_3)\,dx'dy_3,$$
implies that 
\begin{equation}\label{homo1}\int_\Omega\nabla_{x'}\tilde p(x')\cdot\varphi'(x',y_3)\,dx'dy_3=-{\alpha|\partial T'|\over |Y'|}\int_\Omega \mathcal{M}_{\partial T'}[\hat u'](x',y_3)\cdot \tilde \varphi'(x',y_3)\,dx'dy_3\,.
\end{equation}

In order to obtain the homogenized system (\ref{effective1}), we introduce the auxiliary problem 
$$\left\{\begin{array}{l}
\displaystyle -\Delta_{y'}\chi(y')=-{|\partial T'|\over |Y_f'|}\mathcal{M}_{Y'_f}[\hat u](x',y_3),\quad\hbox{in }Y_f',\\
\noame\displaystyle
{\partial \chi\over \partial n}=\hat u,\quad \hbox{on }\partial T',\\
\noame\displaystyle
\mathcal{M}_{Y'_f}[\chi]=0,\\
\noame\displaystyle
 \chi(y) \quad Y'-\hbox{periodic},
\end{array}\right.\,$$
for a.e. $(x',y_3)\in \Omega$, which has a unique solution in $H^1(Y'_f)$ (see Chapter 2, Section 7.3 in Lions and Magenes \cite{LiMa}). Using this auxiliary problem, we conclude that 
\begin{equation}\label{igual_medias}
\int_\Omega \mathcal{M}_{\partial T'}[\hat u]\cdot \tilde \varphi\,dx'dy_3=\int_\Omega \mathcal{M}_{Y_f'}[\hat u]\cdot \tilde \varphi\,dx'dy_3\,,
\end{equation}
which together with (\ref{homo1}) and $\mathcal{M}_{\partial T'}[\hat u_3]=0$ gives 
$$\mathcal{M}_{Y_f'}[\hat u'](x',y_3)=-{|Y'|\over \alpha|\partial T'|}\nabla_{x'}\tilde p(x')\,,$$ and $$\mathcal{M}_{Y'_f}[\hat u_3]=0,$$
which together with (\ref{propertygrau0}) gives 
\begin{equation*}
\tilde u'(x',y_3)= -{|Y'_f|\over \alpha|\partial T'|}\nabla_{x'}\tilde p(x'), \quad \tilde u_3(x',y_3)=0.
\end{equation*}
This together with the definition of $\theta$ and $\mu_1$, implies (\ref{effective1}). 

{\it Step 2.}  For $-1\leq \gamma<1$. First, we prove that $\tilde p$ does not depend on $y_3$. Let $\tilde \varphi=(0,\varepsilon^{2}\tilde \varphi_3) \in \mathcal{D}(\Omega)^3$ be a test function in (\ref{fv_unfolding}). Reasoning as Step 1 and by the second {\it a priori} estimate (\ref{a3-2}), the convergence (\ref{convPtilde1-case2}) and (\ref{convUgorrofrontera2}), passing to the limit we deduce that $\tilde p$ does not depend on $y_3$.

Let $\tilde \varphi=(\varepsilon\varphi'(x',y_3),\varepsilon\tilde \varphi_3(x')) \in \mathcal{D}(\Omega)^3$ be a test function in (\ref{fv_unfolding}). Reasoning as Step 1 and by the second {\it a priori} estimate (\ref{a3-2}), the convergences (\ref{Hf}), (\ref{convPtilde1-case2}) and (\ref{convUgorrofrontera2}), passing to the limit we have
\begin{equation*}
\begin{array}{l}
\displaystyle
-\int_{\Omega}\tilde p(x')\, {\rm div}_{x'}\,\tilde \varphi'(x',y_3)dx'dy_3+{\alpha \over |Y'|}\int_{\omega \times \partial T'}\int_0^1\hat u'(x') \cdot \tilde \varphi'(x',y_3) dx'd\sigma(y')dy_3\\
\noame \displaystyle=\int_{\Omega}f'(x') \cdot \tilde \varphi'(x',y_3)dx'dy_3+{1 \over |Y'|}\int_{\omega \times \partial T'}\int_0^1 g'(y')\cdot \tilde \varphi'(x',y_3)dx'd\sigma(y')dy_3,
\end{array}
\end{equation*}
and
\begin{equation*}
{\alpha \over |Y'|}\int_{\omega \times \partial T}\hat u_3(x') \, \tilde \varphi_3(x') \,dx'd\sigma(y)=0,
\end{equation*}
which implies that $\hat u_3=0$.

Taking into account that 
$$\int_{\omega \times \partial T'}\int_0^1\hat u'(x') \cdot \tilde \varphi'(x',y_3) dx'd\sigma(y')dy_3=|\partial T'|\int_\Omega \hat u'(x')\cdot \tilde \varphi'(x',y_3)\,dx'dy_3,$$
implies that 
\begin{equation*}
\begin{array}{l}
\displaystyle
\int_\Omega\nabla_{x'}\tilde p(x')\cdot\varphi'(x',y_3)\,dx'dy_3+{\alpha|\partial T'|\over |Y'|}\int_\Omega \hat u'(x')\cdot \tilde \varphi'(x',y_3)\,dx'dy_3\\
\noame \displaystyle=\int_{\Omega}f'(x') \cdot \tilde \varphi'(x',y_3)dx'dy_3+{|\partial T'| \over |Y'|}\int_{\Omega }\mathcal{M}_{\partial T'} [g']\cdot \tilde \varphi'(x',y_3)dx'dy_3,
\end{array}
\end{equation*}
which together with (\ref{propertygrau2}) gives (\ref{effective2}) after integrating the vertical variable $y_3$ between $0$ and $1$. 

{\it Step 3.}  For $\gamma \ge 1$. For all  $\hat \varphi(x',y)\in \mathcal{D}(\omega;C_{\sharp}^\infty(Y)^3)$, we choose $\hat \varphi_\varepsilon(x)=\hat \varphi(x',x'/\varepsilon,y_3)$ as test function in (\ref{fv_unfolding}).  Then we get the following variational formulation:
\begin{equation}\label{form_var_step3_1}\begin{array}{l}
\displaystyle
{\mu\over \varepsilon^2} \int_{\omega\times Y_f}D_{y} \hat u_\varepsilon:D_{y'} \hat \varphi\, dx'dy-\int_{\omega\times Y_f}\hat p_\varepsilon\,{\rm div}_{x'}\,\hat \varphi'\, dx'dy-{\varepsilon^{-1}}\int_{\omega\times Y_f}\hat p_\varepsilon\,{\rm div}_{y}\,\hat \varphi\, dx'dy\\
\noame 
\displaystyle
+\alpha {\varepsilon^{\gamma-1}}\int_{\omega \times \partial  T}\hat u_\varepsilon \cdot \hat \varphi\, dx'd\sigma(y) =\int_{\omega\times Y_f}f'_\varepsilon\cdot \hat \varphi'\,dx'dy+{\varepsilon^{-1}}\int_{\omega \times \partial  T} \tilde g' \cdot \hat\varphi'\, dx'd\sigma(y)+O_\varepsilon.
\end{array}
\end{equation}
First, we remark that thanks to (\ref{esti_gorro3}), there exists $\hat p\in L^2(\omega\times Y_f)$ such that  $\varepsilon^2\hat p_\varepsilon$ converges weakly to $\hat p$ in $L^2(\omega\times Y_f)$. Now, we prove that $\hat p$ does not depend on $y$. For that, we consider $\varepsilon^3\hat \varphi_{\varepsilon}$ in the previous formulation, and passing to the limit by (\ref{convUgorrofrontera3}) and (\ref{convUgorro3Derivada}), we get 
$$\int_{\omega\times Y_f}\hat p\, {\rm div}_y \hat \varphi\,dx'dy_3=0,$$
which shows that $\hat p$ does not depend on $y$.

Now, we consider $\hat \varphi$ with ${\rm div}_y\hat \varphi=0$ in $\omega\times Y_f$  (which is necessary because $\hat u$ satisfies   (\ref{incomp_div_hatu})), and similarly we define $\hat\varphi_\varepsilon$. Then, taking $\varepsilon^2\hat\varphi_\varepsilon$, the variational formulation (\ref{form_var_step3_1})  is the following:
\begin{equation}\label{form_var_step3_2}\begin{array}{l}
\displaystyle
{\mu} \int_{\omega\times Y_f}D_{y} \hat u_\varepsilon:D_{y'} \hat \varphi\, dx'dy-\varepsilon^2\int_{\omega\times Y_f}\hat p_\varepsilon\,{\rm div}_{x'}\,\hat \varphi'\, dx'dy\\
\noame 
\displaystyle
+\alpha {\varepsilon^{\gamma+1}}\int_{\omega \times \partial  T}\hat u_\varepsilon \cdot \hat \varphi\, dx'd\sigma(y) =\varepsilon^2\int_{\omega\times Y_f}f'_\varepsilon\cdot \hat \varphi'\,dx'dy+{\varepsilon}\int_{\omega \times \partial  T} \tilde g' \cdot \hat\varphi'\, dx'd\sigma(y)+O_\varepsilon.
\end{array}
\end{equation}
Reasoning as {\it Step 1}, and using the convergences (\ref{Hf}), (\ref{convUgorrofrontera3}), (\ref{convUgorro3Derivada}) and the convergence of $\hat p_\varepsilon$, passing to the limit we have
\begin{equation}\label{form_var_step3_3}\begin{array}{l}
\displaystyle
{\mu} \int_{\omega\times Y_f}D_{y} \hat u:D_{y} \hat \varphi\, dx'dy-\int_{\omega\times Y_f}\hat p(x')\,{\rm div}_{x'}\,\hat \varphi'\, dx'dy=0.
\end{array}
\end{equation}
By density, this equation holds for every function $\hat \varphi(x',y)\in L^2(\omega;H_{\sharp}^1(Y)^3)$ such that ${\rm div}_y\hat \varphi=0$. This implies that there exists $\hat q(x',y)\in L^2_{\sharp}(\omega\times Y_f)$ such that (\ref{form_var_step3_3}) is equivalent to the following problem:
$$\left\{\begin{array}{l}
\displaystyle-\mu\Delta_y \hat u+ \nabla_{y}\hat q=-\nabla_{x'}\hat p\quad\hbox{in }\omega\times Y_f,\\
\noame\displaystyle
{\rm div}_{y}\hat u=0
\quad\hbox{in }\omega\times Y_f,\\
\noame\displaystyle
{\partial\hat u\over \partial n}=0\quad\hbox{on }\omega\times \partial T,\\
\noame\displaystyle \hat u=0\quad\hbox{on }y_3=0,1,\\
\noame\displaystyle
y'\to \hat u(x',y),\hat q(x',y)\quad Y'-{\rm periodic}\,.
\end{array}\right.
$$
We remark that $\hat p$ is already the pressure $\tilde p$. This can be easily proved by multiplying equation (\ref{fv_unfolding}) by a test function $\varepsilon^2\varphi'(x',y_3)$ and identifying limits.


Finally, we will eliminate the microscopic variable $y$ in the effective problem. Observe that we can easily deduce that $\hat u_3=0$  and $\hat q=\hat q(x',y')$ and moreover, the derivation of (\ref{effective3}) from the previous effective problem is an easy algebra exercise. In fact,  we can write $\int_0^1\hat u(x',y)dy_3={1\over \mu}\sum_{i= 1}^2 \partial_{x_i}\tilde p(x')w^i(y')$  and $\hat q(x',y')={1\over \mu}\sum_{i=1}^2\partial_{x_i}\tilde p(x')q^i(y')$ with $(w^i,q^i)$, $i=1,2$, the solutions  of the local problems (\ref{local_problem_3}), and use property (\ref{propertygrau3}) which involves the functions $\int_0^1\hat u(x',y)dy_3$ and $\int_0^1\tilde u(x',y_3)dy_3$. As well-known, the local problems (\ref{local_problem_3}) are well-posed with periodic boundary conditions, and it is easily checked, by integration by parts, that 
$$A_{ij}={1\over |Y_f|}\int_{Y_f}D_yw^i(y):D_yw^j(y)\,dy=\int_{Y_f}w^i(y)e_j\,dy,\quad i,j=1,2.$$
By definition $A$ is symmetric and positive definite.
\section{Conclusions}\label{sec:conclusions}
The behavior of the flow of Newtonian fluids through periodic arrays of cylinders has been studied extensively, mainly because of its importance in many applications in heat and mass transfer equipment. However, the literature on Newtonian thin film fluid flows through periodic arrays of cylinders is far less complete, although these problems have now become of great practical relevance because take place in a number of natural and industrial processes. 
This paper deals with the modelization by means of homogenization techniques of a thin film fluid flow governed by the Stokes system in a thin perforated domain $\Omega_\varepsilon$ which depends on a small parameter $\varepsilon$. More precisely,  $\Omega_\varepsilon$ has thickness $\varepsilon$ and is perforated by a periodic array of cylindrical obstacles of period $\varepsilon$. 

The main novelty here are the combination of  the mixed boundary condition considered on the obstacles and the thin thickness of the domain. Namely,  a standard (no-slip) condition is imposed on the exterior boundary, whereas a non-standard boundary condition of Robin type which depends on a parameter $\gamma$ is imposed on the interior boundary. This type of boundary condition is  motivated by the phenomenon in which a motion of the fluid appears when a electrical field is applied on the boundary of a porous medium in equilibrium. The main mathematical difficulties of this work are to treat the surface integrals and extend the solution to a fixed domain in order to pass to the limit with respect to the parameter $\varepsilon$.  We overcome the first difficulty by using a version of the unfolding method which let us treat the surface integrals quite easily. Moreover, we need to develop some extension abstract results and adapt them to the case of  thin domain.
 
By means of a combination of  homogenization and reduction of dimension techniques, depending on the parameter $\gamma$, we obtain three modified 2D Darcy type laws which model the behavior of the fluid and include the effect of the surface forces and the measure of the obstacles. We remark that we are not able to prove a divergence condition for the limit averaged fluid flow as obtained if we had considered Dirichlet boundary conditions, which from the mechanical point of view means that some fluid ``dissapear" through the cylinders and so, it is represented by the motion of a compressible fluid.
To conclude,  it is our firm belief that our results will prove useful in the engineering practice, in particular in those industrial applications where the  flow is affected by the effects of the surface forces, the fluid microstructure and the thickness of the domain.

\section*{Acknowledgments}
We would like to thank the referees for their comments and suggestions. Mar\'ia Anguiano has been supported by Junta de Andaluc\'ia (Spain), Proyecto de Excelencia P12-FQM-2466. Francisco Javier Su\'arez-Grau has been supported by Ministerio de Econom\'ia y Competitividad (Spain), Proyecto Excelencia MTM2014-53309-P.

\end{document}